\documentclass[10pt,a4paper]{amsart}
\usepackage{a4wide,amsmath,amsfonts,amssymb,amsthm,mathrsfs,setspace,bm,amsxtra,mathtools,wasysym,textcomp} 
\mathtoolsset{showonlyrefs}
\usepackage[english]{babel}
\usepackage{color}
\usepackage[all]{xy}
\usepackage{verbatim}

\newtheorem{thm}{Theorem}[section]
\newtheorem{prop}[thm]{Proposition}
\newtheorem{lemma}[thm]{Lemma}
\newtheorem{defin}[thm]{Definition}
\newtheorem{cor}[thm]{Corollary}

\numberwithin{equation}{section}
\theoremstyle{definition}
\newtheorem{rem}[thm]{Remark}

\newcommand{\Aut}{\operatorname{Aut}}
\newcommand{\rk}{\operatorname{rk}}
\newcommand{\de}{\mbox{\rm d}}
\newcommand{\Hom}{\operatorname{Hom}}
\newcommand{\End}{\operatorname{End}}

\newcommand{\GL}{\operatorname{GL}}
\newcommand{\im}{\operatorname{Im}}
\newcommand{\Pic}{\operatorname{Pic}}
\newcommand{\Ol}{\mathcal{O}}

\newcommand{\R}{\mathcal{R}}
\newcommand{\A}{\mathcal{A}}

\newcommand{\G}{\mathcal{G}}
\newcommand{\F}{\mathcal{F}}

\newcommand{\M}{\mathcal{M}}
\newcommand{\E}{\mathcal{E}}
\newcommand{\Ss}{\mathcal{S}}
\newcommand{\U}{\mathcal{U}}
\newcommand{\V}{\mathcal{V}}
\newcommand{\W}{\mathcal{W}}
\newcommand{\K}{\mathcal{K}}

\newcommand{\N}{\mathcal{N}}

\newcommand{\id}{\operatorname{id}}

\newcommand{\Gr}{\operatorname{Gr}}
\newcommand{\tot}{\operatorname{Tot}}
\newcommand{\Spec}{\operatorname{Spec}}

\newcommand{\Pk}{P_{\vec{k}}}
\newcommand{\Lk}{L_{\vec{k}}}
\newcommand{\Gk}{G_{\vec{k}}}
\newcommand{\Z}{\mathbb{Z}}

\newcommand{\Com}{\mathbb{C}}
\newcommand{\Pu}{\mathbb{P}^1}
\newcommand{\On}{\Ol_{\Sigma_n}}

\newcommand{\Uk}{\U_{\vec{k}}}
\newcommand{\Vk}{\V_{\vec{k}}}
\newcommand{\Wk}{\W_{\vec{k}}}

\newcommand{\li}{\ell_\infty}

\newcommand{\Ext}{\operatorname{Ext}}

\newcommand{\lra}{\longrightarrow}

\newcommand{\Stab}{\operatorname{Stab}}
\newcommand{\ch}{\operatorname{ch}}
\newcommand{\Hilb}{\operatorname{Hilb}}
\newcommand{\tol}{\longrightarrow}
\newcommand{\toi}{\stackrel{\sim}{\tol}}

\newcommand{\Q}{\mathcal{Q}}
\newcommand{\Fl}{\operatorname{Fl}}
\newcommand{\lmt}{\longmapsto}
\newcommand{\tr}{\operatorname{tr}}

\newcommand{\Rep}{\operatorname{Rep}}
\newcommand{\sharpf}{\mathrm{fr}}
\definecolor{rosso}         {rgb}{1.00 , 0.00 , 0.00}

\renewcommand\rightmark{{\scshape Moduli spaces of framed sheaves and quiver varieties}}

\title{Moduli spaces of framed sheaves and quiver varieties} 
\date{\today}

\begin{document}

\maketitle
\begin{center}{\sc Claudio Bartocci,$^\P$ Valeriano Lanza,$^{\S}$ and Claudio L. S.~Rava$^\P$ } \\[10pt]  \small 
$^\P$Dipartimento di Matematica, Universit\`a di Genova (Italy)\\[3pt]
  $^{\S}$Instituto de Matem\'{a}tica, Estat\'{i}stica e Computa{\c c}\~{a}o Cient\'{i}fica, Universidade Estadual de Campinas (Brazil)\\ 
\end{center}
\bigskip
\begin{quote}\small {\sc Abstract.} In the first part of this paper we provide a survey of some fundamental results about moduli spaces of framed sheaves on smooth projective surfaces. In particular, we outline a result by Bruzzo and Markushevich, and discuss a few significant examples. The moduli spaces of framed sheaves on $\mathbb{P}^2$, on multiple blowup of  $\mathbb{P}^2$ are described in terms of ADHM data and, when this characterization is available, as quiver varieties.\\
The second part is devoted to a detailed study of framed sheaves on the Hirzebruch surface $\Sigma_n$ in the case when the invariant expressing the necessary and sufficient condition for the nonemptiness of moduli spaces attains its minimum (what we call the ``minimal case''). Our main result is that, under this assumption, the corresponding moduli space is isomorphic to a Grassmannian (when $n=1$), or to the direct sum of $n-1$ copies of the cotangent bundle of a Grassmannian (when $n\geq 2$). Finally, by slightly generalizing a construction due to Nakajima, we prove that these moduli spaces admit a description as quiver varieties. 
 
\end{quote}

\bigskip
\section{Introduction}
In 1984 Donaldson \cite{Don84} showed that the moduli space of $SU(r)$-instantons on $\mathbb{R}^4$ is isomorphic to the moduli space of rank $r$ holomorphic vector bundles on the complex projective plane $\mathbb{P}^2$ which are {\em framed} on a line at infinity, i.e.~which are trivial when restricted  to that line and have a fixed trivialization there.  His proof exploited, on the one hand, the ADHM construction \cite{ADHM}, and, on the other, the monadic description of vector bundles over $\mathbb{P}^2$ \cite{barth77, hulek79, ocs80}. According to the first correspondence the moduli space of instantons  on $\mathbb{R}^4$ is interpreted as a 
hyper-K\"ahler quotient, while, according to the latter, the moduli space of framed holomorphic vector bundles on $\mathbb{P}^2$ is realized as a symplectic GIT quotient. The two constructions (as it was remarked by Donaldson) turn out to be equivalent thanks to a result by Kirwan \cite{K84} based on previous work by Kempf and Ness \cite{KN78}.

The generalization of the ADHM construction to the case of ALE spaces \cite{KNa90} prompted Nakajima to introduce the notion of {\em quiver variety} \cite{Na94}. 
Very broadly speaking, a Nakajima's quiver variety associated with a quiver $\Q$ can be described as a coarse moduli space
of semistable representations of an auxiliary quiver $\overline{\Q^{\sharpf}}$ concocted from $\Q$ (see \S \ref{quiversection} for precise statements). Under very mild assumptions 
(Theorem \ref{ginzburgtheorem} $=$ \cite[Theorem~5.2.2.(ii)]{Gin}), every Nakajima's quiver variety is a smooth quasi-projective variety and carries a holomorphic symplectic structure obtained through a Hamiltonian reduction. Besides the above motivations, the interest of such varieties lies in the fact that they have been used by Nakajima \cite{Na94, Na98} ``to give a geometric construction of universal enveloping algebras of Kac-Moody Lie algebras and of all irreducible integrable (e.g., finite dimensional) representations of those algebras'' \cite[p.~143]{Gin}.  For more information on this thriving subject see \cite{Schiff} and references therein.

 
After Donaldson's pioneering result, a great deal of work has been directed to the general study of moduli spaces of framed sheaves on (stacky) complex surfaces \cite{L93, HL95a, HL95b, N02a, N02b, BM, SAL, BS}.  Given a complex variety $X$, an effective divisor $D$ on $X$ and a torsion-free sheaf $\mathcal{F}_D$ on $D$, a {\em framed sheaf} is a pair $(\mathcal{E}, \theta)$, where 
$\mathcal{E}$ is a torsion-free sheaf on $X$ and $\theta\colon \mathcal{E}\vert_D \toi \mathcal{F}_D$ is an isomorphism. 
Huybrechts  and Lehn \cite{HL95a, HL95b}, working in a slightly more general setting, introduced a stability condition giving rise to fine moduli spaces of framed sheaves. On smooth projective surfaces, as shown by Bruzzo and Markushevich \cite{BM}, fineness can be ensured  by replacing  the stability condition by the one of ``good framing'' (see Definition \ref{goodnessdef}).  More precisely, they strengthened previous results by Nevins \cite{N02a, N02b} proving that, if $D$ is a big and nef curve on $X$ and $\F_D$ is a good framing sheaf on $D$, then for any class $\gamma \in H^\bullet(X,\mathbb Q)$ there exists a (possibly empty) quasi-projective scheme ${\mathcal {M}}_X(\gamma)$ that is a fine moduli space of $(D, \mathcal{F}_D)$-framed sheaves on $X$ with Chern character $\gamma$ 
(\cite[Thm~3.1 ]{BM} $=$ Thm~\ref{thmBM} herein).
%
%

However, the aforementioned theorem by Bruzzo and Markushevich does not provide any information either about the nonemptiness of the moduli space of framed shaves in question, or about its local and global geometric structure. The investigation of specific examples, therefore, retains its importance.
The original construction for framed bundles on $\mathbb{P}^2$ was extended by King \cite{Ki89} to framed holomorphic vector bundles on the blowup of $\mathbb{P}^2$ at a point and by Buchdahl  \cite{Bu93} to multiple blowups; Henni \cite{He} generalised Buchdahl's results to the non-locally free sheaves.  
The case of framed torsion-free sheaves on $\mathbb{P}^2$ was first described by Nakajima in his influential lectures \cite{Na99}. This represents a key example, because the corresponding moduli spaces admit an ADHM description in terms of linear data arising as moment map equation of a holomorphic symplectic quotient and are Nakajima's quiver varieties. Moreover, these spaces are desingularizations  of the moduli spaces of ideal instantons over $\mathbb{R}^4$ \cite{Na99}, so that they can be exploited to compute Nekrasov's partition function \cite{Nek03, BFMT} and to perform the so-called instanton counting 
(i.e., the technique consisting in the use of localization for the equivariant cohomology of the moduli space acted upon by a suitable torus to compute invariants of the moduli space and of the surface) \cite{NY05}. It should be pointed out that, contrary to the case of $\mathbb{P}^2$, the moduli spaces constructed by King, Buchdhal and Henni in the above cited papers have not been given a description as quiver varieties (even in a broader sense than Nakajima's original one). Whether or not this is possible seems to be a challenging problem.



Moduli spaces of framed sheaves on Hirzebruch surfaces $\Sigma_n$ have been studied in the papers \cite{BBR, BBLR} (see also \cite{Ra11, La15}). In particular, a fine moduli space $\M^{n}(r, a,c)$ parameterizing isomorphism classes of framed sheaves $\E$ on $\Sigma_n$, which have Chern character $\textrm{ch}(\E) = (r, aE, -c -\frac{1}{2} na^2)$ and are trivial, with a fixed trivialization, on  a ``line at infinity'',  is constructed by means of a monadic approach (Theorem \ref{thmEMon}). One can prove that the space $\M^{n}(r, a,c)$ is nonempty if and if the inequality $$c \geq C_{\text{m}}(n,a)= \frac{1}{2} na(1-a)$$
 is satisfied (Theorem \ref{mainthm}). We shall refer to the case when equality holds as the ``minimal case''.

In the rank 1 case the moduli space of framed sheaves on 
$\Sigma_n$ can be naturally identified with the Hilbert scheme of points on the total space of the line bundle $\Ol_{\Pu}(-n)$ over $\Pu$, in an analogous way to that
by which the moduli space of rank 1 framed sheaves on $\mathbb{P}^2$ is identified with $\operatorname{Hilb}^c(\mathbb{C}^2)$.
The schemes $\operatorname{Hilb}^c (\tot(\Ol_{\Pu}(-n)))$ admit a realization in terms of generalized ADHM data and turn out to be irreducible connected components of the moduli spaces of representations of suitable quivers. So, they  are quiver varieties, although {\it not} in Nakajima's sense (except for $n=2$), but in a more general one (see \cite{Gin} and Section \ref{quiversection}). When $n=2$, 
$\tot(\Ol_{\Pu}(-2))$ is the ALE space $A_1$, and indeed our description coincides with that one obtained by Kuznetsov in \cite{kuz}.

While, in general, the spaces $\M^{n}(r, a,c)$ seem to resist to be described as quiver varieties, such a description is achievable, in addition to the rank 1 case, also for the minimal case. Actually, as proved in Proposition \ref{corMshMq}, the schemes $\M^{n}(r, a,C_{\text{m}}(n,a))$ are quiver varieties, which slightly generalize those used by Nakajima \cite{Na94, Na96} to represent partial flag varieties (see Proposition \ref{propCotFlag}).

\bigskip
The present paper is organized as follows. Sections \ref{sectionBM}, \ref{quiversection}, and \ref{MFSsection} are chiefly of an expository nature.
In \S \ref{sectionBM} we outline Bruzzo and Markushevich's theorem about the existence of a fine moduli space for good framings sheaves, while in \S \ref{quiversection} we recall some basics facts about quivers, their representations, and quiver varieties. In \S \ref{MFSsection} we review the construction of the moduli spaces of framed sheaves on $\mathbb{P}^2$ and on $\mathbb{P}^2$ blown-up at $n$ distinct points: albeit both admit a description in terms of ADHM data,  only in the first case a characterization as quiver varieties is always available (see Remark \ref{remarkADHM}). Section \ref{SecMon} is devoted to 
summarize the basics of the construction of the moduli space of framed sheaves on Hirzebruch surfaces in terms of monads, as worked out in \cite{BBR, BBLR}. 
The last two sections contain our main results about the minimal case. In \S \ref{sectionminimalcase} we prove (Theorem \ref{thm:minimal}) that $\M^{1}(r, a,C_{\text{m}})$ is isomorphic to the Grassmannian  $\operatorname{Gr}(a,r)$, while, for $n\geq 2$,  $\M^{n}(r, a, C_{\text{m}})$ is isomorphic to the total space of the vector bundle $T^\vee\operatorname{Gr}(a,r)^{\oplus n-1}$. Moreover, we show (\S \ref{geometricremarks}) that, heuristically, the Grassmannian variety parameterizes the (isomorphism classes of) framings, while the fibres of the direct sum of (copies of) the cotangent bundles classify the sheaves away from the line at infinity. Finally, in \S 
\ref{Section-Nakajima-Flag} we provide a description of the moduli spaces $\M^{n}(r, a, C_{\text{m}})$ as quiver varieties.

\bigskip
{\bf Acknowledgements.}  We warmly thank our friend Ugo Bruzzo: not only he co-authored the two papers whose contents are summarised in Section \ref{SecMon}, but gave substantial contributions to the results of Section \ref{sectionminimalcase} as well. We are grateful to the anonymous referee, whose suggestions and remarks helped us to substantially improve the original, very much shorter, version of this paper.
This work was partially supported by the PRIN ``Geometria delle variet\`a algebriche'' and by the University of Genoa's research grant
``Aspetti matematici della teoria dei campi interagenti e quantizzazione per deformazione''. V.L. is supported by the
FAPESP post-doctoral grant number 2015/07766-4.

\bigskip{\bf Conventions.} A {\em scheme} is a separated scheme of finite type over $\mathbb{C}$; a {\em variety} is a reduced scheme. A {\em sheaf} is a coherent sheaf.

\bigskip\section{Moduli spaces of framed sheaves}\label{sectionBM}

We summarize the main definitions and results of Sections 2 and 3 of \cite{BM}. Let $X$ be a smooth projective variety of dimension $n\geq 2$, $D \subset X$ an effective divisor, and $\mathcal{F}_D$ a torsion-free sheaf on $D$. 

\begin{defin}\label{definframedsheaf} A torsion-free sheaf $\mathcal{E}$ on $X$ is \emph{$(D, \mathcal{F}_D)$-framable} if there is an $\mathcal{O}_X$-epimorphism $\mathcal{E} \to \mathcal{F}_D$ restricting to an isomorphism $\mathcal{E}\vert_{D} \toi \mathcal{F}_D$. A \emph{$(D, \mathcal{F}_D)$-framed sheaf $(\mathcal{E}, \theta)$ on $X$} is a pair consisting of a $(D, \mathcal{F}_D)$-framable sheaf $\mathcal{E}$ and a framing isomorphism $\theta\colon \mathcal{E}\vert_{D} \toi \mathcal{F}_D$. Two framed sheaves 
$(\mathcal{E}, \theta)$ and $(\mathcal{E}', \theta')$ are isomorphic if there is an $\mathcal{O}_X$-isomorphism $\Lambda\colon \mathcal{E} \toi \mathcal{E}'$ such that
$\theta' \circ \Lambda\vert_{D}= \theta$. 
\end{defin}

Let $H$ be a polarization on $X$. For any sheaf $\mathcal{E}$ on $X$, we denote by $P^H(\mathcal{E})$ the Hilbert polynomial $P^H(\mathcal{E})(k) =
\chi(\mathcal{E}\otimes \mathcal{O}_X(kH))$. If $\rk \E > 0$, for any nonnegative real number $\eta$ , the $\eta$-slope of $\E$ is the quantity 
$$\mu^H_{\eta}(\E) = \frac{c_1(\E)\cdot H^{n-1} - \eta}{\rk \E}\,.$$
For $\eta =0$ we recover, of course, the ordinary notion of slope, and so we write  $\mu^H_{0}(\E) = \mu^H(\E)$.
A framed sheaf $(\mathcal{E}, \theta)$ is said to be {\em $(H, \delta)$-stable}, where $\delta$ is a positive real number,
if any subsheaf $\Ss \subset \E$ with $0 < \rk  \Ss \leq \rk \E$ satisfies one of the following inequalities:
\begin{align*}
&\mu^H (\Ss) < \mu^H_\delta(\E) \quad \text{if}\, \Ss \, \text{is contained in the kernel of the restriction} \  \E \to \E\vert_{D}\,; \\
&\mu^H_\delta (\Ss) < \mu^H_\delta(\E) \quad \text{otherwise.}
\end{align*}
By applying the results proved by Huybrechts and Lehn \cite{HL95a, HL95b} in the setting of stable pairs, one can show that the isomorphism classes of $(H, \delta)$-stable framed sheaves $(\mathcal{E}, \theta)$ with fixed Hilbert polynomial $P^{H}(\mathcal{E})$ form a fine moduli space which is a quasi-projective scheme.

A fine moduli space for framed sheaves can be also constructed through a different approach. Instead of imposing (as above) a stability condition on the sheaf that is to be framed, one requires the framing divisor and the framing sheaf to have the {\it goodness} property stated in the following definition.
\begin{defin}\label{goodnessdef}  \cite[Def.~2.4]{BM} A divisor $D$ on $X$ is a \emph{good framing divisor} if it can be written as the sum of prime divisors $D_i$ with positive integer coefficients $a_i$, i.e.~$D = \sum a_i D_i$, and there exists a big and nef divisor of the form $\sum b_i D_i$, with $b_i \geq 0$. A sheaf $\F_D$ on $D$ is a \emph{good framing sheaf} if it is locally free and there exists a nonnegative real number $C < (\rk \F_D)^{-1} D^2 \cdot H^{n -2}$ such that, for any locally free subsheaf $\Ss_D \subset \F_D$ of constant positive rank, one has
\begin{equation*} \frac{\operatorname{deg} c_1(\Ss_D)}{\rk \Ss_D} \leq \frac{\operatorname{deg} c_1(\F_D)}{\rk \F_D} + C\,.
\end{equation*}
\end{defin}
If $D$ is a good framing divisor, and if we fix any vector bundle $\G_{D}$ on $D$, any polarization $H$ and any polynomial $P \in {\mathbb Q}[t]$, then the set $\mathcal M$ of 
torsion-free sheaves $\mathcal{E}$  such that   $\E|_{D}\simeq\G_{D}$ and $P^H(\mathcal{E})= P$  is {\em bounded}, that is to say, there exists a scheme $\mathfrak L$ along with a sheaf  ${\mathbf E}$ over $X\times \mathfrak L$ such that, for any $\mathcal{E} \in \mathcal M$, there is a closed point $q\in {\mathfrak L}$ 
and an isomorphism $\mathcal{E}  \simeq {\mathbf E}_{X\times \{q\}}$ \cite[Thm~2.5]{BM}. Let us restrict ourselves to the case  
$\dim X =2$. The fact that the set $\mathcal M$ is bounded implies that the Castelnuovo-Mumford regularity\footnote{Recall that, for a sheaf $\E$ on a polarized surface $(X,H)$, the Castelnuovo-Mumford regularity $\rho^H(\E)$ is the minimal integer $m$ such that $h^i (X, \E \otimes {\mathcal O}_X((m-i)H)) = 0$ for all $i >0$.}
$\rho^H(\E)$ is uniformly bounded over all of $\mathcal M$;
therefore,  by \cite[Lemma 1.7.9]{HLbook}, there is a constant $\tilde C$ 
(depending solely on $H$, $\F_D$ and $P$) such that $\mu^H(\Ss) \leq \mu^H(\E) + \tilde C$ for all $\E \in {\mathcal M}$ and for all nonzero subsheaves 
$\Ss \subset \E$. If the good framing divisor $D$ is a big and nef curve, then the divisor $H_N= H + ND$ is ample for any $N >0$:  a direct computation shows that there exists a positive integer $N^\ast$ such that the range of positive real numbers $\delta$, for which all the framed sheaves $\E \in \mathcal M$ are 
$(H_{N^\ast}, \delta)$-stable, is nonempty. Since, as said above, $(H_{N^\ast}, \delta)$-stable framed sheaves constitute fine moduli spaces which are quasi-projective schemes, the following result is proved.
\begin{thm}\label{thmBM} \cite[Thm~3.1]{BM} Let $X$ be a smooth projective surface, $D$ a big and nef curve on $X$, and $\F_D$ a good framing sheaf on $D$. Then for any class $\gamma \in H^\bullet(X,\mathbb Q)$, there exists a (possibly empty) quasi-projective scheme ${\mathcal {M}}_X(\gamma)$ that is a fine moduli space of $(D, \mathcal{F}_D)$-framed sheaves on $X$ with Chern character $\gamma$.
\end{thm}
In the case where the framing divisor $D$ is a smooth and irreducible curve and $D^2 >0$, a semistable vector bundle on $D$ is a good framing sheaf with $C=0$.
Hence, Theorem \ref{thmBM} entails the following result.
\begin{cor} \cite[Cor.~3.3]{BM}\label{corBM}
Let $X$ be a smooth projective surface, $D$ a smooth, irreducible, big and nef curve, and $\mathcal{F}_D$ a semistable vector bundle on $D$. 
For any class $\gamma\in H^\bullet(X, \mathbb {Q})$,  there exists a (possibly empty) quasi-projective scheme ${\mathcal {M}}_X(\gamma)$ that is a fine moduli space  
of $(D, \mathcal{F}_D)$-framed sheaves on $X$ with Chern character $\gamma$.
\end{cor}

\bigskip\section{Generalities on quivers and quiver varieties}\label{quiversection}
In the subsequent sections we shall provide a few explicit realizations of moduli spaces of framed sheaves as quiver varieties. With this end in mind, we recall here some basic facts
about quiver representations and quiver varieties (see \cite{Gin} for details). 

A quiver $\mathcal{Q}$ is a finite oriented graph, given by a set of vertices $I$ and a set of arrows $E$.
The path algebra $\Com\mathcal{Q}$ is the $\Com$-algebra with basis the paths in $\mathcal{Q}$ and with a product given by the composition of paths whenever possible, zero otherwise. Usually  one includes among the generators of $\Com\mathcal{Q}$ a complete set of orthogonal idempotents $\{e_{i}\}_{i\in I}$: this can be considered a subset of $E$ by regarding $e_{i}$ as a loop of ``length zero'' starting and ending at the $i$-th vertex. A  (complex) representation of a quiver $\mathcal{Q}$ is a pair $(V,X)$, where $V= \bigoplus_{i\in I}V_{i}$ is an $I$-graded complex vector space and $X=(X_{a})_{a\in E}$ is a collection of linear maps such that $X_{a}\in\Hom_{\Com}(V_i,V_j)$
 whenever the arrow $a$ starts at the vertex $i$ and terminates at the vertex $j$.
We say that a representation $(V,X)$ is supported by $V$, and denote by $\operatorname{Rep}(\mathcal{Q},V)$ the space of representations of $\mathcal{Q}$ supported by 
$V$. Morphisms and direct sums of representations are defined in the obvious way; it can be shown that the abelian category of complex representations of $\Q$ is equivalent to the category of left $\Com\mathcal{Q}$-modules. A subrepresentation of a given representation $(V,X)$ is a pair $(S,Y)$, where $S$ is an $I$-graded subspace of $V$ which is preserved by the linear maps $X$, and $Y$ is the restriction of $X$ to $S$.  We  consider only finite-dimensional representations. If $\dim_{\Com} V_{i}=v_{i}$, a representation $(V,X)$ of $\mathcal{Q}$ is said to be $\vec{v}$-dimensional, where $\vec{v}=(v_{i})_{i\in I}\in\mathbb{N}^{I}$.  

More generally one can define the representations of a quotient algebra $B=\Com\mathcal{Q}/J$, for some ideal $J$ of the path algebra $\Com\mathcal{Q}$. In this case it is customary to call the pair $(\Q,J)$ a \emph{quiver with relations}. The representations of $B$ are the subset of the representations $(V,X)$ of $\mathcal{Q}$, whose linear maps $X=(X_{a})_{a\in E}$ satisfy the relations given by the elements of $J$. The abelian category of complex representations of $B$ is equivalent to the category of left $B$-modules.  We denote by $\operatorname{Rep}(B,\vec{v})$ the space of representations of $B$ supported by a given $\vec{v}$-dimensional vector space $V$. 
There is a natural action of $G_{\vec{v}}=\prod_i\GL(v_i)$ on $\operatorname{Rep}(B,\vec{v})$ given by change of basis. One would like to consider the space of isomorphism classes of $\vec{v}$-dimensional representations of $B$, but unfortunately in most cases this space is   ``badly behaved''. To overcome this drawback, adopting A.~King's strategy \cite{king}, one introduces a notion of (semi)stability depending on the choice of a parameter $\vartheta\in\mathbb{R}^{I}$. Let us recall that the $\vartheta$-{\emph slope} $\mu_\vartheta(V,X)$ of a nontrivial representation $(V,X)$ of $B$ is the ratio
$$\mu_\vartheta(V,X)= \frac{\sum_{i\in I} \vartheta_i v_i}{\sum_{i\in I} v_i}\,.$$
\begin{defin} \label{defstability}
A representation $(V,X)$ of $B$ is said to be $\vartheta$-semistable (resp., $\vartheta$-stable) if, for any proper nontrivial subrepresentation $(S,Y)$, one has $\mu_\vartheta(S,Y) \leq \mu_\vartheta(V,X)$
(resp., strict inequality holds).
\end{defin}
\begin{rem} It may be worth emphasizing that we do not assume  $\sum_{i\in I} \vartheta_i v_i = 0$ as in the original paper \cite{king}, following instead Rudakov's approach \cite{Rud97} (which is actually equivalent to King's one; see also \cite[Remark 2.3.3]{Gin}).
\end{rem}
Let $\operatorname{Rep}(B,\vec{v})^{ss}_\vartheta$ be the subset of $\operatorname{Rep}(B,\vec{v})$ consisting of semistable representations. According to Proposition 5.2 of \cite{king}, the coarse moduli space of $\vec{v}$-dimensional $\vartheta$-semistable representations of $B$ is  the GIT quotient 
$$\M(B,\vec{v})_{\vartheta}=\operatorname{Rep}(B,\vec{v})^{ss}_\vartheta/\!/_{\!\vartheta}G_{\vec{v}}\,.$$ 
If $\vec{v}$ is a primitive vector, then the open subset $\M^{s}(B,\vec{v})_{\vartheta} \subset \M(B,\vec{v})_{\vartheta}$ consisting of stable representations
makes up a fine moduli space \cite[Proposition 5.3]{king}.

A useful construction in studying representations of quivers is that of \emph{framed quiver}. Given a quiver $\Q$ with vertex set $I$ and arrow set $E$,
its framed quiver $\Q^{\mathrm{fr}}$ is defined as the quiver whose vertex set is $I\sqcup I'$, where $I'$ is a copy of $I$ with a fixed bijection $i\to i'$
and whose arrow set $E^{\mathrm{fr}}$ is obtained by adding for every $i\in I$ new arrows $i \stackrel{d_i}{\lra} i'$ to $E$. 

In view  of Remark \ref{hilbertasquiver} and of the results of Section \ref{Section-Nakajima-Flag}, we find it convenient to introduce also the notion of \emph{generalized framed}
(GF) quiver. Given a quiver $\Q$ as above, we denote by 
$\Q^{\mathrm{gfr}}$ \emph{an} associated quiver whose vertex set is $I\sqcup I'$, where $I'$ is a copy of $I$ with a fixed bijection $i\to i'$, and whose arrow set $E^{\mathrm{gfr}}$ is obtained by adding to $E$ new arrows 
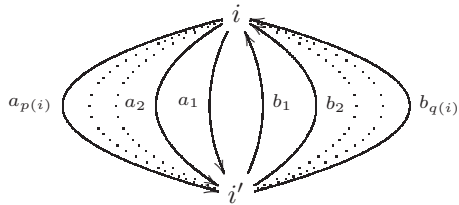
\begin{equation}
\xymatrix@R+3em{
i \ar@/_10pt/[d]_-{a_{1}} \ar@/_30pt/[d]_-{a_{2}} \ar@/_45pt/@{.}[d] \ar@/_55pt/@{.}[d] \ar@/_65pt/[d]_-{a_{p(i)}} 
\\
 i' \ar@/_10pt/[u]_-{b_{1}} \ar@/_30pt/[u]_-{b_{2}} \ar@/_45pt/@{.}[u] \ar@/_55pt/@{.}[u] \ar@/_65pt/[u]_-{b_{q(i)}}
}
\end{equation}
with $p(i) > 0$ and $q(i) \geq 0$,  for all $i\in I$.
Of course, when $p(i)=1$ and $q(i)=0$ for all $i\in I$, one recovers the standard definition of $\Q^{\mathrm{fr}}$.

%

 A representation of $\Q^{\mathrm{gfr}}$ is supported by $V\oplus W$, where $V$ and $W$ are finite-dimensional $I$-graded vector spaces. If $\dim_{\Com} V_{i}=v_{i}$ and $\dim_{\Com} W_{i}=w_{i}$, a representation $(V\oplus W,X)$ of $\mathcal{Q}^{\mathrm{gfr}}$ is said to be $(\vec{v},\vec{w})$-dimensional, where $\vec{v}=(v_{i})_{i\in I}\in\mathbb{N}^{I}$ and $\vec{w}=(w_{i})_{i\in I}\in\mathbb{N}^{I}$. 
 We  denote by $\operatorname{Rep}(\Q^{\mathrm{gfr}},\vec{v},\vec{w})$ the space of representations of $\Q^{\mathrm{gfr}}$ supported by a fixed $(\vec{v},\vec{w})$-dimensional vector space $V\oplus W$. In the sequel, we shall assume that $\vec{w} \neq 0$.

Analogously to the unframed case, one can define the path algebra $\Com\K$ of a GF quiver $\K$, and consider the representations of the quotient algebra $\Lambda=\Com\K/L$, for any given ideal $L$  of $\Com \K$. Also the notion of subrepresentation is completely analogous to the unframed case. However, the representation space $\operatorname{Rep}(\Lambda,\vec{v},\vec{w})$ is regarded just as a $G_{\vec{v}}$-variety, where the group $G_{\vec{v}}$ acts by change of basis on $(V_{i})_{i\in I}$, whilst the action of the group  $G_{\vec{w}}$ is ignored.

The previous notion of (semi)stability is extended to the representations of a GF quiver by slightly modifying a result due to Crawley-Boevey \cite[p.~261]{CrBo}. Let $\Q$ be a quiver with vertex set $I$, $\Q^{\mathrm{gfr}}$ an associated GF quiver, and $J$ an ideal of the algebra $\Com\Q^{\mathrm{gfr}}$.
\begin{lemma}
\label{lmIsoCB}
For all dimension vectors $(\vec{v},\vec{w})\in\mathbb{R}^{I\sqcup I'}$, there exist a quiver $\Q^{\vec{w}}$, with vertex set $I\sqcup\{\infty\}$, and an ideal $J^{\vec{w}}$ of the algebra $\Com\Q^{\vec{w}}$ such that there is a $G_{\vec{v}}$-equivariant isomorphism
\begin{equation}
\operatorname{Rep}\left(\Com\Q^{\mathrm{gfr}}/J ,\vec{v},\vec{w}\right)\simeq\operatorname{Rep}\left(\Com\Q^{\vec{w}}/J^{\vec{w}}, \vec{v},1\right)\,.
\label{eqIsoCB0}
\end{equation}
\end{lemma}
The quiver $\Q^{\vec{w}}$ of Lemma \ref{lmIsoCB} is built by adding to $\Q$ a vertex at $\infty$, and, for any $i\in I$, a number of arrows from the vertex $i$ to the vertex $\infty$ and viceversa equal, respectively, to $w_i\,p(i)$ and $w_i\, q(i)$.
Note that a representation of $\Q^{\vec{w}}$ is supported by a vector space $V\oplus V_{\infty}$, where $V$ is an $I$-graded vector space, whilst $V_{\infty}$ is a vector space associated with the vertex $\infty$. Such a representation is said to be $(\vec{v},v_{\infty})$-dimensional, where $v_{\infty}=\dim_{\Com}V_{\infty}$. The group $G_{\vec{v}}$ acts on the right-hand side of \eqref{eqIsoCB0} by change of basis on $(V_{i})_{i\in I}$. Let $\vec{v}\in\mathbb{N}^{I}$ and $\vartheta\in\mathbb{R}^{I}$ be, respectively, a dimension vector and a stability parameter for a representation of $\Q$. One defines a stability parameter $\widehat{\vartheta}=(\vartheta,\vartheta_{\infty})\in\mathbb{R}^{I\sqcup\{\infty\}}$ for a $(\vec{v},1)$-dimensional representations of $\Q^{\vec{w}}$ by setting $\vartheta_{\infty}=-\vartheta\cdot\vec{v}$. 
\begin{defin}
\label{defFrStab}
A $(\vec{v},\vec{w})$-dimensional representation of $\Com\Q^{\mathrm{gfr}}/J$ is said to be $\vartheta$-semistable (resp., stable) if and only if its image through the isomorphism \eqref{eqIsoCB0} is $\widehat{\vartheta}$-semistable (resp., stable).
\end{defin}

Given a quiver $\Q$ with vertex set $I$, for any  $\vec{v},\vec{w}\in\mathbb{N}^{I}$, $\lambda\in\Com^{I}$ and $\vartheta\in\mathbb R^{I}$, one can define 
the associated \emph{Nakajima quiver variety} $\N_{\lambda,\vartheta}(\mathcal{Q},\vec{v},\vec{w})$ \cite{Na94}. 
The main steps of the construction of $\N_{\lambda,\vartheta}(\mathcal{Q},\vec{v},\vec{w})$ can be summarized as follows (see \cite{Gin} for further details):
\begin{enumerate}
 \item[(i)] One considers the framed quiver $\Q^{\sharpf}$ and the associated \emph{double} $\overline{\Q^{\sharpf}}$. The latter  has the same vertex set as $\Q^{\sharpf}$, and for any arrow $i\stackrel{a}{\lra}j$ in $E^{\sharpf}$, with $i,j\in I\sqcup I'$,  one adds an \emph{opposite} arrow $j\stackrel{a^{*}}{\lra}i$. It is easy to see that, for all dimension vectors $(\vec{v},\vec{w})$, there is an isomorphism
\begin{equation}
\Rep(\overline{\Q^{\sharpf}},\vec{v},\vec{w})\simeq T^{\vee} \Rep(\Q^{\sharpf},\vec{v},\vec{w})\,.
\end{equation}
As a consequence of that, $\Rep(\overline{\Q^{\sharpf}},\vec{v},\vec{w})$ carries a canonical holomorphic symplectic form $\tilde{\omega}$:
\begin{equation}
\tilde{\omega}=\tr\left(\sum_{a\in E}\de X_{a}\wedge\de X_{a^{*}}+\sum_{i\in I}\de X_{d_{i}}\wedge \de X_{d_{i}^{*}}\right)\,.
\label{eq-tildeomega}
\end{equation}
\item[(ii)] Let $\mathfrak{g}_{\vec{v}}=\bigoplus_{i\in I}\End_{\Com}(\Com^{v_{i}})$ be the Lie algebra associated with $G_{\vec{v}}$. The group $G_{\vec{v}}$ acts naturally on $\Rep(\overline{\Q^{\sharpf}},\vec{v},\vec{w})$ if one regards $\GL(v_i)$ as the group of   automorphisms of the vector space associated with the $i$-th vertex (of the original quiver $\mathcal{Q}$). Since this action is symplectic, one can introduce a moment map $\mu\colon \Rep(\overline{\Q^{\sharpf}},\vec{v},\vec{w})\to\mathfrak{g}_{\vec{v}}^\ast\simeq\mathfrak{g}_{\vec{v}}$, given by
\begin{equation}
(V\oplus W, X)\mapsto \sum_{a\in E}\left(X_a\circ X_{a^\ast}-X_{a^\ast}\circ X_a\right)+\sum_{i \in I}X_{d_{i}^{*}}\circ X_{d_{i}}\,.
\end{equation}
This gives rise to a moment element, which we call again $\mu$:
\begin{equation}
\mu=\sum_{a\in E}[a,a^{*}]+\sum_{i\in I} d_{i}^{*}d_{i}\in\Com\overline{\Q^{\sharpf}}\,.
\end{equation}
It is easy to see that $\mu$ admits a decomposition $\mu=(\mu_{i})_{i\in I}$, where $\mu_{i}\in e_{i}(\Com\overline{\Q^{\sharpf}})e_{i}$.
\item[(iii)] The \emph{framed preprojective algebra $\Pi_{\lambda}(\Q)$ of $\mathcal{Q}$ with parameter $\lambda$} is defined as the quotient $\Com\overline{\Q^{\sharpf}}/J$, where $J$ is the ideal of $\Com\overline{\Q^{\sharpf}}$ generated by the elements $\{\mu_i-\lambda_i\}_{i\in I}$. The fibre
$$\mu^{-1}\left(\sum_{i \in I}\lambda_i \mathbf{1}_{v_i}\right)  \subset  \operatorname{Rep}(\overline{\Q^{\sharpf}},\vec{v},\vec{w})$$
  is the space of $(\vec{v},\vec{w})$-dimensional representations of $\Pi_{\lambda}(\Q)$, which we shall denote by  $\operatorname{Rep}(\Pi_{\lambda}(\Q),\vec{v},\vec{w})$.
\item[(iv)] The quiver $\overline{\Q^{\sharpf}}$ can be regarded as a GF quiver associated with the double $\overline{\Q}$ of $\Q$. By applying to the representations of $\Pi_{\lambda}(\Q)$ the notion of (semi)stability introduced in Definition \ref{defFrStab},  one introduces the space $\Rep(\Pi_{\lambda}(\Q),\vec{v},\vec{w})_{\vartheta}^{\mathrm{ss}}$ of $\vartheta$-semistable representations and defines the
Nakajima quiver variety $\N_{\lambda,\vartheta}(\mathcal{Q},\vec{v},\vec{w})$ as the
 quotient 
 $$\N_{\lambda,\vartheta}(\mathcal{Q},\vec{v},\vec{w})=\Rep(\Pi_{\lambda}(\Q),\vec{v},\vec{w})_{\vartheta}^{\mathrm{ss}}/\!/_{\!\vartheta}G_{\vec{v}}\,.$$
\end{enumerate}

The symplectic form \eqref{eq-tildeomega} is $G_{\vec{v}}$-invariant, so that it induces a $G_{\vec{v}}$-invariant Poisson structure $\{-,-\}\sptilde$ on $\Rep(\Pi_{\lambda},\vec{v},\vec{w})_{\vartheta}^{\mathrm{ss}}$. As the quotient $\N_{\lambda,\vartheta}(\mathcal{Q},\vec{v},\vec{w})$ is a Hamiltonian reduction, it inherits a Poisson structure $\{-,-\}$. Theorem~5.2.2.(ii) in \cite{Gin} provides a sufficient condition for the smoothness of $\N_{\lambda,\vartheta}(\mathcal{Q},\vec{v},\vec{w})$ and for the nondegeneracy of $\{-,-\}$. To state this result we need to introduce some notation and a definition. Given a quiver $\Q$, with vertex set $I$ and arrow set $E$, we define its \emph{adiacency matrix} $A_{\Q}$ and its \emph{Cartan matrix} $C_{\Q}$ as follows: $A_{\Q}=(a_{ij})_{i,j\in I}$, where $a_{ij}=\sharp\{ \text{arrows in $E$ going from $j$ to $i$}\}$, and $C_{\Q}=2\bm{1}_{I}-A_{\overline{\Q}}$. For all dimension vectors $\vec{v}\in\mathbb{N}^{I}$ let $R_{\Q}(\vec{v})$ be the set
of vectors
\begin{equation}
R_{\Q}(\vec{v})=\left\{\vec{u}\in\Z^{I}\setminus\{0\}\left|
\begin{gathered}
(C_{\Q}\vec{u})\cdot \vec{u}\leq 2\\
0\leq u_{i}\leq v_{i}\qquad\text{for all $i\in I$}
\end{gathered}
\right.\right\}\,.
\end{equation}
Given a vector $\vec{s}\in\mathbb{R}^{I}$, we put $\vec{s}^{\bot}=\left\{\vec{t}\in\mathbb{R}^{I}\left|\vec{t}\cdot\vec{s}=0\right.\right\}$.
\begin{defin}\label{defin-v-reg}
Given a dimension vector $\vec{v}\in\mathbb{N}^{I}$, a pair of parameters $(\lambda,\vartheta)\in\Com^{I}\oplus\mathbb{R}^{I}\simeq \mathbb{R}^{3}\otimes_{\mathbb{R}}\mathbb{R}^{I}$ is said to be \emph{$\vec{v}$-regular} if and only if 
\begin{equation}
(\lambda,\vartheta)\in \mathbb{R}^{3}\otimes_{\mathbb{R}}\mathbb{R}^{I} )\setminus (\bigcup_{\vec{u}\in R_{\Q}(\vec{v})}\mathbb{R}^{3}\otimes_{\mathbb{R}}\vec{u}^{\bot})\,.
\end{equation}
\end{defin}
\begin{thm}\cite[Theorem~5.2.2.(ii)]{Gin} \label{ginzburgtheorem}
Let $\Q$ be a quiver with vertex set $I$,
let $\vec{v},\vec{w}\in\mathbb{N}^{I}$ be dimension vectors such that $\vec{w}\neq0$, and let $(\lambda,\vartheta)\in\Com^{I}\oplus\mathbb{R}^{I}$ be a $\vec{v}$-regular pair of parameters. Then all $\vartheta$-semistable representations of $\Pi_{\lambda}$ are $\vartheta$-stable, the variety $\N_{\lambda,\vartheta}(\mathcal{Q},\vec{v},\vec{w})$ is smooth and connected of dimension $2\vec{w}\cdot\vec{v}-(C_{\Q}\vec{v})\cdot\vec{v}$, and the Poisson structure $\{-,-\}$ is nondegenerate.
\end{thm}
Thus, under the hypotheses of Theorem \ref{ginzburgtheorem}, the symplectic form \eqref{eq-tildeomega} descends to a symplectic form $\omega$ on $\N_{\lambda,\vartheta}(\mathcal{Q},\vec{v},\vec{w})$. It is easy to see that a pair of parameters $(0,\vartheta)$, with $\vartheta\neq0$ and $\vartheta_{i}\geq0$ for all $i\in I$, is $\vec{v}$-regular for all dimension vectors $\vec{v}\in\mathbb{N}^{I}$.

\bigskip\bigskip

\section{Two examples of moduli spaces of framed sheaves}\label{MFSsection}

\subsection{Framed sheaves on $\mathbb{P}^2$} \label{P2section}
In the terminology of Section \ref{sectionBM}, we take the line at infinity $\ell_\infty=\{[z_0:z_1:0]\in\mathbb{P}^2\}$ as framing divisor, and the trivial sheaf 
$\Ol_{\ell_\infty}^{\oplus r}$ as framing sheaf. So, an $(\li, \Ol_{\li}^{\oplus r})$-framed  (hereafter, simply \emph{framed}) sheaf on $\mathbb{P}^2$ is a pair $(\mathcal{E},\theta)$, where $\mathcal{E}$ is a torsion-free sheaf of rank $r$ such that $\mathcal{E}|_{\ell_\infty}\simeq\Ol_{\ell_\infty}^{\oplus r}$, and $\theta\colon\mathcal{E}|_{\ell_\infty}\stackrel{\sim}{\longrightarrow}\Ol_{\ell_\infty}^{\oplus r}$ is an isomorphism. By Corollary \ref{corBM} we know that for any $\gamma\in H^\bullet(\mathbb{P}^2,\mathbb{Q})$ there exists a fine moduli space $\mathcal{M}_{\mathbb{P}^2}(\gamma)$ for framed sheaves on $\mathbb{P}^2$ with Chern character $\gamma=(\gamma_0,\gamma_1,\gamma_2)$. Note, however, that this space is empty whenever $\gamma_1\neq 0$, since the existence of the framing implies the vanishing of the first Chern class.  We denote by $\mathcal{M}(r,c)$ the space $\mathcal{M}_{\mathbb{P}^2}(\gamma)$ with $\gamma=(r,0,c)$.

The following result is due to Nakajima \cite{Na99} and provides an ADHM description for this moduli space (generalising a previous result by Donaldson for framed vector bundles on $\mathbb{P}^2$ \cite[Prop.~1]{Don84}).
\begin{thm}\cite[Thm~2.1]{Na99}\label{thm:Naka}
 There is an isomorphism of algebraic varieties between $\M(r,c)$ and $N(r,c)/\GL(c,\Com)$, where $N(r,c)$ is the quasi-affine variety of quadruples
 \[
  (B_1,B_2,i,j)\in\End(\Com^c)^{\oplus 2}\oplus\Hom(\Com^r,\Com^c)\oplus\Hom(\Com^c,\Com^r) 
 \]
 satisfying the conditions
 \begin{itemize}
  \item[(i)] $[B_1,B_2]+ij=0$;
  \item[(ii)] there exists no proper subspace $S\subsetneq \Com^c$ such that $B_\alpha(S)\subseteq S$ ($\alpha=1,2$) and $\im i\subseteq S$,
 \end{itemize}
and the group $\GL(c,\Com)$ acts on $N(r,c)$ by means of the formula
\[
 g\cdot(B_1,B_2,i,j)=(gB_1g^{-1},gB_2g^{-1},gi,jg^{-1})\,.
\]
\end{thm}

\begin{rem}
Theorem \ref{thm:Naka}  allows us
to interpret the space $\M(r,c)$ as the Nakajima quiver variety $\N_{0,-1}(\mathcal{L}_{1},c,r)$, where $\mathcal{L}_{1}$ is the so-called \emph{Jordan quiver}, having one vertex and one loop at this vertex. In other words, $\M(r,c)$ can be also viewed as the moduli space of $(-1)$-semistable representations of the framed preprojective algebra $\Pi_{0}(\mathcal{L}_1)$; note that $\overline{\mathcal{L}_1^{\mathrm{fr}}}$ is the quiver

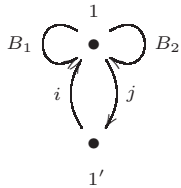
\[
  \xymatrix@R-2.3em{
  \mbox{\scriptsize$1$}\\
  \bullet\ar@(ul,dl)_{B_1}\ar@(ur,dr)^{B_2}\ar@/^2ex/[dddd]^{j}\\ \\ \\ \\
  \bullet\ar@/^2ex/[uuuu]^{i}\\
  \mbox{\scriptsize\phantom{'}$1'$}
} 
\]
It is easy indeed to recognize the moment map equation in the condition (i) of Theorem \ref{thm:Naka}, while condition (ii) ensures precisely that the representations we are considering are $(-1)$-semistable.
\end{rem}

The proof of Theorem \ref{thm:Naka} relies on the fact that any torsion-free sheaf $\E$ on $\mathbb{P}^2$ which is trivial along $\ell_\infty$ and has Chern character $\ch(\E) = (r, 0, c)$ is isomorphic to the cohomology of the monad 
\begin{equation}\label{monadP2}
M(a,b) :\qquad  \xymatrix{
  0\ar[r]&V\otimes\Ol_{\mathbb{P}^2}(-1)\ar[r]^-a&\widetilde{W}\otimes\Ol_{\mathbb{P}^2}\ar[r]^-b&V'\otimes\Ol_{\mathbb{P}^2}(1)\ar[r]&0\,,
 }
\end{equation}
where $V$, $\widetilde{W}$, and $V'$ are complex vector spaces of dimension, respectively, $c$, $2c+r$, and $c$. So, one has $\mathcal{E}\simeq\ker b/\im a$.

Applying Theorem \ref{thm:Naka} to the rank $1$ case, one gets ADHM data for the Hilbert scheme of points of $\Com^2$. Indeed, the double dual $\E^{\ast\ast}$ of $\E$ is locally free and has vanishing first Chern class, so that it is isomorphic to the structure sheaf $\Ol_{\mathbb{P}^2}$. As a consequence, since $\E$ is trivial along $\ell_{\infty}$, the mapping carrying $\E$ to the schematic support of $\Ol_{\mathbb{P}^2}/ \E$  yields an isomorphism
\begin{equation} 
\M(1,c) \simeq \operatorname{Hilb}^c (\mathbb{P}^2 \setminus \ell_{\infty}) = \operatorname{Hilb}^c (\Com^2)\,.
\end{equation} 
In this particular case the stability condition implies that $j=0$ \cite[Prop.~2.8 (1)]{Na99}, so that the description of $\Hilb^c(\Com^2)$ can be given in terms of triples $(B_1,B_2,i)$, and condition (i) of Theorem \ref{thm:Naka} can be rephrased by saying that $B_1$ and $B_2$ are commuting matrices.
\begin{rem}
Let $D$ be the divisor $\{\infty\}\times  \mathbb{P}^1 \cup  \mathbb{P}^1\times \{ \infty\}$ on the surface $\mathbb{P}^1 \times \mathbb{P}^1$. The moduli space
of $(D, \mathcal{O}_D)$-framed sheaves on $\mathbb{P}^1 \times \mathbb{P}^1$ of rank $r$ and second Chern class $c$ is isomorphic to $\M(r,c)$. So, there is an action of the group $\Gamma = \Z / n\Z$ on $\M(r,c)$ which is induced by the action of $\Gamma$ on $\mathbb{P}^1 \times \mathbb{P}^1$ given
by the multiplication of the second coordinate of the second $\mathbb{P}^1$ by the $n$-th roots of unity. It can be proved \cite{Bis, FR} that a connected component of $\M^\Gamma(r,c)$, the $\Gamma$-equivariant locus inside $\M(r,c)$, is isomorphic to the moduli space ${\mathcal P}_{\underline{c}}$ of parabolic bundles on $\mathbb{P}^1 \times \mathbb{P}^1$, where $\underline{c} = (c_0, \dots , c_{n-1})$ is a partition of $c$. Starting from the quiver construction of $\M(r,c)$, Finkelberg and Rybnikov \cite{FR} provided
a description of ${\mathcal P}_{\underline{c}}$ as a quiver variety (the quiver in question is a chainsaw quiver). Takayama \cite{Tak} exploited this description to establish an isomorphism between moduli spaces of solutions to Nahm's equations over the circle and moduli spaces of locally free parabolic sheaves over 
$\mathbb{P}^1 \times \mathbb{P}^1$.
\end{rem}

\subsection{Framed sheaves on multi-blow-ups of the projective plane}
We denote by $\widetilde{\mathbb{P}}^2$ the complex projective plane blown-up at $n$ distinct points $p_1,\dots,p_n\notin \ell_\infty$; let
$\varpi\colon \widetilde{\mathbb{P}}^2 \lra \mathbb{P}^2$ be the canonical projection. The Picard group of $\widetilde{\mathbb{P}}^2$ is freely generated on $\Z$ by the class $H$, the divisor of the pullback of the generic line in $\mathbb{P}^2$, and by the classes $\{E_i\}_{i=1}^n$, $E_i$ being the exceptional divisor corresponding to the blow-up at $p_i$. Analogously to the case of $\mathbb{P}^2$, 
we take $\tilde{\ell}_\infty =\varpi^{-1}(\ell_\infty)$ as framing divisor and the trivial sheaf 
$\Ol_{\tilde{\ell}_\infty}^{\oplus r}$ as framing sheaf.

Corollary \ref{corBM} ensures that there exists a fine moduli space $\mathcal{M}_{\widetilde{\mathbb{P}}^2}(\gamma)$ for framed sheaves $(\E, \theta)$ on $\widetilde{\mathbb{P}}^2$ with Chern character $\gamma=(r ,\gamma_1, -c +\frac{1}{2}\gamma_1^2) \in H^\bullet(\widetilde{\mathbb{P}}^2,\mathbb{Q})$. Notice that the first Chern class of every torsion-free sheaf which is trivial on $\tilde{\ell}_\infty$ has no component along $H$;  
hence, $\mathcal{M}_{\widetilde{\mathbb{P}}^2}(\gamma)$ is empty whenever $\gamma_1\cdot H \neq 0$. When 
$\gamma_1 = \sum_{i=1}^na_iE_i$, the space $\mathcal{M}_{\widetilde{\mathbb{P}}^2}(\gamma)$ will be denoted by $\widetilde{\M}(r;a_1,\dots,a_n;c)$.
An explicit description of such a space was given by Henni in \cite{He}. The particular result for locally free sheaves had been previously proved first by King \cite[Thm~3.3.2]{Ki89} in the case when $n=1$ and $c_1=0$ and then extended by Buchdal \cite[Prop.~1.10]{Bu04} to general values of $n$ and $c_1$.

\begin{thm} \cite[Prop.~2.20]{He}
A torsion-free sheaf $\E$ on  $\widetilde{\mathbb{P}}^2$ which is trivial along $\tilde{\ell}_{\infty}$ and has Chern character $\operatorname{ch}(\E) = (r, \sum_{i=1}^na_iE_i, -c -\frac{1}{2} \sum_{i=1}^na_i^2)$ is isomorphic to the cohomology of a monad 
\begin{equation}
\xymatrix{
0 \ar[r] & 
\bigoplus_{s=0}^nK_s\otimes\Ol_{\widetilde{\mathbb{P}}^2}(-H+E_s) \ar[r]^-{\alpha} & W\otimes\Ol_{\widetilde{\mathbb{P}}^2} \ar[r]^-{\beta} &
\bigoplus_{s=0}^nL_s\otimes\Ol_{\widetilde{\mathbb{P}}^2}(H-E_s) \ar[r] & 0
}\,,
\end{equation}
where $E_0:=0$ and $\{K_s\}_{s=0}^n$, $W$, $\{L_s\}_{s=0}^n$ are complex vector spaces with
 \begin{align*}
 \dim K_0&=c+\frac{1}{2}\sum_{i=1}^na_i(a_i+1)=:k\ ,\quad\dim L_0=c+\frac{1}{2}\sum_{i=1}^na_i(a_i-1)=:l\,,\\
 \dim W&= 2(n+1)k-2\sum_{i=1}^na_i+r\ ,\quad\dim K_s=k-a_s \ ,\quad \dim L_s=k\quad \text{($s=1,\dots,n$)}\,.
 \end{align*}
\end{thm}
Before providing the ADHM description for $\widetilde{\M}(r;a_1,\dots,a_n;c)$, we notice that, since $p_1,\dots,p_n\notin\ell_\infty=\{[z_0:z_1:0]\in\mathbb{P}^2\}$, they all belong to the standard affine chart $U_2=\{z_2\neq 0\}$; we denote by $(p_i^0,p_i^1)$ the affine coordinates of $p_i$ inside $U_2$.

A space of ADHM data for $\widetilde{\M}(r;a_1,\dots,a_n;c)$ is given by the quasi-affine variety $H(r;a_1,\dots,a_n;c)$ in
\[ 
 \Hom(\Com^k,\Com^l)^{\oplus 3}\oplus\bigoplus_{s=1}^n\Hom(\Com^{k-a_s},\Com^k)\oplus\bigoplus_{s=1}^n\Hom(\Com^{k-a_s},\Com^l)\oplus\Hom(\Com^k,\Com^r)\oplus\Hom(\Com^r,\Com^l)
\]
characterized by the following conditions:\\
for any point $(A,C_0,C_1;B_1,\dots,B_n;B_1',\dots,B_n';e;f) \in H(r;a_1,\dots,a_n;c)$,\\
(i) the $(l+nk)\times(l+nk)$ matrix
\[
 M:=
 \begin{pmatrix}
   A & B_1'&B_2'&\cdots&B_n'\\
   \mathbf{1}_{k}&B_1&0&\cdots&0\\
   \mathbf{1}_{k}&0&B_2&\cdots&0\\
   \vdots&\vdots&\vdots&\ddots&\vdots\\
   \mathbf{1}_{k}&0&0&\cdots&B_n
  \end{pmatrix}
  \]
 is invertible;
 
 \noindent
(ii) the $(l+nk)\times(l+nk)$ matrices
 \[
 Q_j:=
  \begin{pmatrix}
   -C_j & p_1^jB_1'&p_2^jB_2'&\cdots&p_n^jB_n'\\
   p_1^j\mathbf{1}_{k}&p_1^jB_1&0&\cdots&0\\
   p_2^j\mathbf{1}_{k}&0&p_2^jB_2&\cdots&0\\
   \vdots&\vdots&\vdots&\ddots&\vdots\\
   p^j_n\mathbf{1}_{k}&0&0&\cdots&p_n^jB_n
  \end{pmatrix}\qquad \text{for $j=0,1$}\,,
 \]
satisfy the equation
 \[
  [Q_0M^{-1}Q_1-Q_1M^{-1}Q_0]_l^k+fe=0\,,
 \]
where the notation $[\,\star\,]_l^k$ denotes the block of the matrix $\star$ formed by the first $l$ rows and $k$ columns.

Let $G$ be the subgroup of $\GL(l+nk, \Com)\times\GL(l+nk,\Com)$ whose elements $(h,g)$ are of the following form
\begin{equation}\label{group-form}
h=\operatorname{diag}(h_{0},h_{1},\cdots,h_{n}),\quad\quad
g=\left(\begin{array}{lllll}g_{0}&g_{1}&g_{2}&\cdots&g_{n}\\
0&h^{-1}_{0}&0&\cdots&0\\
0&0&h^{-1}_{0}&\cdots&0\\
\vdots&\vdots&\vdots&\ddots&\vdots\\
0&0&0&\cdots&h^{-1}_{0}\\
\end{array}\right)\,,
\end{equation}
where $h_{0}\in\GL(k,\Com)$, $g_{0}\in\GL(l,\Com)$, $h_{s}\in\GL(k-a_{s}, \Com)$ and $g_{s}\in\operatorname{Mat}_\Com(l\times k)$ for $s=1,\dots,n$. We define a $G$-action on $H(r;a_1,\dots,a_n;c)$ indirectly, by means of the formulas
\begin{equation}
\left\{
\begin{array}{rcl}
M & \to & gMh\\
Q_{j} & \to & gQ_{j}h\quad\text{for}\quad j=0,1\\
e & \to & eh_{0}\\
f & \to & g_{0}f
\end{array}
\right.\qquad\qquad (h,g)\in G\,.
\end{equation}

Finally, we get the following result (cf.~also \cite[Thm~3.4.1]{Ki89} and \cite[\S 3]{Bu04}).
\begin{thm}\cite[Thm~6.1]{He} The variety $H(r;a_1,\dots,a_n;c)$ is a locally trivial principal $G$-bundle over $\widetilde{\M}(r;a_1,\dots,a_n;c)$; in particular, there is an isomorphism of smooth algebraic varieties  
$\widetilde{\M}(r;a_1,\dots,a_n;c) \simeq H(r;a_1,\dots,a_n;c)/G\,.$
\end{thm}

\begin{rem}\label{remarkADHM}
Contrary to the case of $\mathbb{P}^2$, these ADHM data cannot be exploited, as they stand, to interpret the space $\widetilde{\M}(r;a_1,\dots,a_n;c)$, for general values of the invariants, as a quiver variety. The main difficulty in this respect is that the group $G$ cannot be regarded, in general, as the group of isomorphisms of representations of any quiver (cf.~\cite[proof of Lemma 6.7]{He}). Nevertheless, in the case $n=1$ this difficulty can be overcome: indeed, as shown by Henni \cite[\S 2.2.1]{HePhD}, the $G$-principal bundle $H(r;a_{1};c)\lra \widetilde{\M}(r;a_1;c)$ admits a reduction to a $\GL(k,\Com)\times\GL(l,\Com)$-principal bundle. It is likely that the space $\widetilde{\M}(r;a_1;c)$ can be embedded into a quiver variety: a hint in this direction comes from the case of rank $1$ sheaves. In fact there is an isomorphism $\widetilde{\M}(r;a_1;c)\simeq \M^{1}(r,a_{1},c)$, where the latter is the moduli space of framed sheaves on the first Hirzebruch surface (see Section \ref{SecMon}). When $r=1$, we know  that $\widetilde{\M}(1;a_1;c)\simeq \widetilde{\M}(1;0;c)$ is isomorphic to a connected component of a quiver variety \cite[Theorem 4.1]{BBLR}.
\end{rem}

\bigskip\section{Framed sheaves on Hirzebruch surfaces}\label{SecMon}

The $n$-th Hirzebruch surface $\Sigma_n$ can be defined as the projective closure of the total space of the line bundle $\mathcal{O}_{\mathbb{P}^1}(-n)$. 
There is a natural ruling $\Sigma_n\longrightarrow\Pu$ whose fibre
determines a class $F \in \Pic(\Sigma_n)$. Let $H$ and $E$ be the classes of  sections squaring, respectively, to $n$ and $-n$: one can prove that $\Pic(\Sigma_n)$ is freely generated on $\Z$ by $H$ and $F$; we put  $\On(p,q) = \On(pH+qF)$. It should be also recalled that  $\Sigma_n$ is a Poisson surface \cite[Remark 2.5]{BaMa}.

In what follows we assume $n>0$, in view of the fact that we wish to choose as framing divisor a curve in the class $H$. The class $H$ is big if and only if $n>0$ (for one has $H^{2}=n$), so that if $n=0$ a curve in $H$ does not satisfies the hypotheses of Corollary \ref{corBM}.

Let  $\ell_{\infty}\simeq\Pu$ be a ``line at infinity'' which belongs to the class $H$ and does not intersect $E$. 
An $(\ell_{\infty},  \Ol_{\li}^{\oplus r})$-framed sheaf (or, for brevity's sake, a {\em framed sheaf}) on $\Sigma_n$ is a pair $(\E, \theta)$, where $\E$ is a rank $r$ torsion-free sheaf trivial along  $\ell_{\infty}$ and $\theta \colon \E\vert_{\ell_{\infty}}\stackrel{\sim}{\longrightarrow}\Ol_{\li}^{\oplus r}$ is an isomorphism. Notice that the condition of being trivial at infinity implies $c_{1}(\E)\propto E$.

By Corollary \ref{corBM}, there exists a fine moduli space  $\M^{n}(r,a,c)= \M_{\Sigma_n}(\gamma)$ parameterizing isomorphism classes of framed sheaves $(\E, \theta)$ on $\Sigma_n$ with Chern character $\textrm{ch}(\E) =\gamma = (r, aE, -c - \frac{1}{2} na^2)$.
We assume that the framed sheaves are normalized in such a way that $0\leq a\leq r-1$.

\begin{thm} \cite[Cor.~4.6]{BBR}\label{thmEMon}
A torsion-free sheaf $\E$ on  $\Sigma_n$ which is trivial along $\ell_{\infty}$ and such that $\textrm{ch}(\E) = (r, aE, -c -\frac{1}{2} na^2)$ is isomorphic to the cohomology of a monad 
\begin{equation}
\xymatrix{
M(\alpha,\beta):&0 \ar[r] & \Uk \ar[r]^-{\alpha} & \Vk \ar[r]^-{\beta} & \Wk \ar[r] & 0
}\,, \label{fundamentalmonad}
\end{equation}
where $\vec{k} = (n,r,a,c)$ and 
\begin{equation} 
\Uk:=\On(0,-1)^{\oplus k_1},\quad
\Vk:=\On(1,-1)^{\oplus k_2} \oplus \On^{\oplus k_4},\quad
\Wk:=\On(1,0)^{\oplus k_3}\,,
\end{equation}
with
\begin{equation} 
 k_1=c+\dfrac{1}{2}na(a-1),\quad
k_2=k_1+na,\quad
k_3=k_1+(n-1)a,\quad
k_4=k_1+r-a\,.
\label{k_i}
\end{equation}
\end{thm}

The set $\Lk$ of pairs in $\Hom(\Uk,\Vk)\oplus\Hom(\Vk,\Wk)$ fitting into the complex  \eqref{fundamentalmonad} and such that the cohomology of the complex is
torsion-free and trivial at infinity is a smooth algebraic variety. 
We now wish to parameterize isomorphism classes of framed sheaves $(\E, \theta)$, with $\E$ isomorphic to the cohomology of \eqref{fundamentalmonad}. To this aim, we first introduce a principal $\GL(r,\Com)$-bundle $\Pk$ over $\Lk$, whose fibre at a point $(\alpha,\beta)$ is  identified with the space of framings for $\E$.
Next we take the action on $\Pk$ of the algebraic group $\Gk=\Aut(\Uk)\times\Aut(\Vk)\times\Aut(\Wk)$. The action is free, and  the quotient $\Pk/\Gk$ is a smooth algebraic variety \cite[Thm~5.1]{BBR}. This variety can be identified with the moduli space $\M^{n}(r,a,c)$ by constructing a universal family. One of the advantages
of the monadic description is the possibility of obtaining a necessary and sufficient condition for the nonemptiness of  $\M^{n}(r,a,c)$.

\begin{thm}\cite[Thm~3.4]{BBR}\label{mainthm}
The moduli space $\M^{n}(r,a,c)$ is nonempty if and only if 
\begin{equation}\label{NEC}
c + \frac{1}{2} na(a-1) \geq 0\,,
\end{equation}
Whenever this condition is satisfied, it is a smooth, irreducible algebraic variety of dimension $2rc + (r-1) na^2$.
\end{thm}

\subsection{Hilbert schemes of points of $\tot\Ol_{\mathbb{P}^1}(-n)$}
The rank 1 case is especially important. Indeed, from the assumption $r=1$ it follows that $a=0$ (in our normalization), so that one can argue as in the case of $\mathbb{P}^2$ (see Subsection \ref{P2section}), and conclude that
\begin{equation} 
\M^{n}(1,0,c) \simeq \operatorname{Hilb}^c (\Sigma_n \setminus \ell_{\infty}) = \operatorname{Hilb}^c (\tot(\Ol_{\Pu}(-n)))\,.
\end{equation}
Morevover, in this case, the monadic representation gives rise to a genuine ADHM description of the Hilbert scheme of points  $\operatorname{Hilb}^c (\tot(\Ol_{\Pu}(-n)))$, given as follows.

We denote by $P^{n}(c)$ the subset of the vector space $\End(\Com^{c})^{\oplus n+2}\oplus\Hom(\Com^{c},\Com)$ whose points $\left(A_1,A_2;C_1,\dots,C_{n};e\right)$ satisfy the following conditions:
\begin{enumerate}
 \item[(P1)]
\begin{equation*}
\begin{cases}
A_1C_1A_{2}=A_2C_{1}A_{1}&\qquad\text{when $n=1$}\\[15pt]
\begin{aligned}
A_1C_q&=A_2C_{q+1}\\
C_qA_1&=C_{q+1}A_2
\end{aligned}
\qquad\text{for}\quad q=1,\dots,n-1&\qquad\text{when $n>1$;}
\end{cases}
\end{equation*} \smallskip
\item[(P2)]
$A_1+\lambda A_2$ is a \emph{regular pencil} of matrices; equivalently, there exists $[\nu_{1},\nu_{2}]\in\Pu$ such that $\det(\nu_1A_1+\nu_2A_2)\neq0$;\smallskip
\item[(P3)]
for all values of the parameters $\left([\lambda_1,\lambda_2],(\mu_1,\mu_{2})\right)\in\Pu\times\Com^{2}$ such that
\begin{equation*}
\lambda_{1}^{n}\mu_{1}+\lambda_{2}^{n}\mu_{2}=0
\end{equation*}
there is no nonzero vector $v\in\Com^c$ such that
\begin{equation*}
\left\{
\begin{array}{l}
C_{1}A_{2}v=-\mu_1v\\
C_{n}A_{1}v=(-1)^n\mu_2v\\
v\in\ker e
\end{array}\right.
\qquad\text{and}\qquad\left(\lambda_2{A_1}+\lambda_1{A_2}\right)v=0\,.
\end{equation*}\smallskip
\end{enumerate}
The space $P^n(c)$ is a space of ADHM data for our Hilbert schemes; indeed, one has:
\begin{thm} \cite[Thm~3.1]{BBLR}\label{3..1BBLR}
 There is an isomorphism of smooth algebraic varieties between $\M^n(1,0,c)$ and $P^n(c)/(\GL(c,\Com)\times\GL(c,\Com))$, where the group $\GL(c,\Com)\times\GL(c,\Com)$ acts on $P^n(c)$ by means of the formula
\[
 (\phi_1,\phi_2)\cdot(A_i,C_j,e)=(\phi_2A_i\phi_1^{-1},\, 
\phi_1C_j\phi_2^{-1},\,
e\phi_1^{-1})\,.
\]
\end{thm}

Notice that, as shown in \cite{BBLR}, there is an open cover of $\operatorname{Hilb}^c(\operatorname{Tot}\mathcal{O}_{\mathbb{P}^1}(-n))$, whose elements are all isomorphic to $\operatorname{Hilb}^c(\mathbb{C}^2)$: as a matter of fact, the restriction of our ADHM data to these open sets coincides with Nakajima's ADHM data for $\operatorname{Hilb}^c(\mathbb{C}^2)$.

\begin{rem}\label{hilbertasquiver} By relying on the ADHM description given in Theorem \ref{3..1BBLR}, it is possible to prove that the spaces
$\operatorname{Hilb}^c(\operatorname{Tot}\mathcal{O}_{\mathbb{P}^1}(-n))$ 
 can be embedded, as irreducible connected components, into moduli spaces of semistable representations of suitable quotients of the path algebras of the GF quivers
 
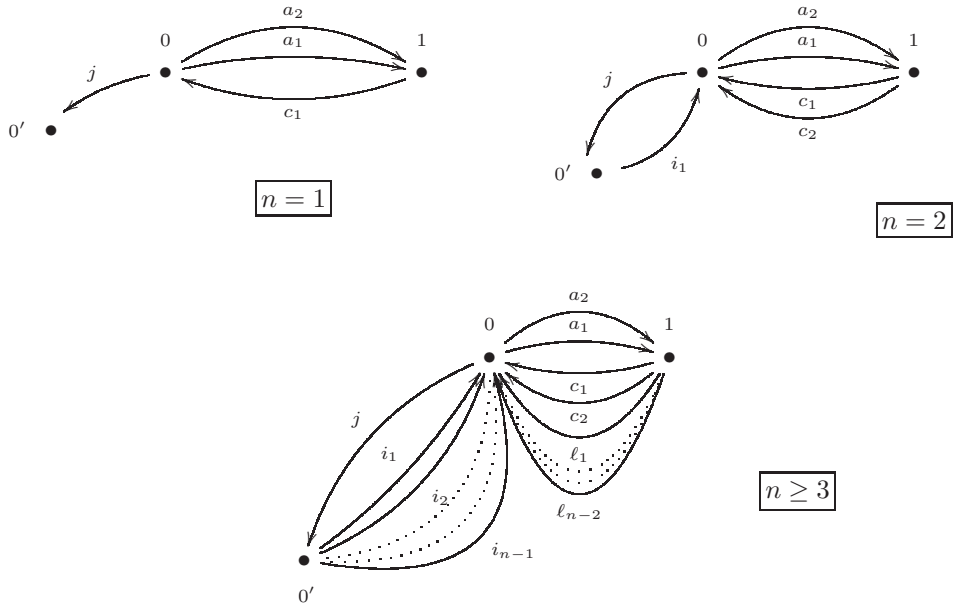
\begin{equation} \label{eqQnfr}
 \xymatrix@R-2.3em{
&\mbox{\scriptsize$0$}&&\mbox{\scriptsize$1$}\\
&\bullet\ar@/_1ex/[ldd]_{j}\ar@/^/[rr]^{a_{1}}\ar@/^4ex/[rr]^{a_{2}}&&\bullet\ar@/^10pt/[ll]^{c_{1}} \\ \\
\mbox{\scriptsize$0'$}\ \ \bullet\ \ &&&\\
&&\\ &&\mbox{\framebox[1cm]{\begin{minipage}{1cm}\centering $n=1$\end{minipage}}}
}
\qquad\qquad
  \xymatrix@R-2.3em{
&\mbox{\scriptsize$0$}&&\mbox{\scriptsize$1$}\\
&\bullet\ar@/_3ex/[ldddd]_{j}\ar@/^/[rr]^{a_{1}}\ar@/^4ex/[rr]^{a_{2}}&&\bullet\ar@/^/[ll]^{c_{1}}
\ar@/^4ex/[ll]^{c_2} \\ \\ \\&&&&& \\ 
\mbox{\scriptsize$0'$}\ \ \bullet\ \ar@/_12pt/[ruuuu]_{i_1}&&&&& \\
&&& \mbox{\framebox[1cm]{\begin{minipage}{1cm}\centering $n=2$\end{minipage}}}
}
\end{equation}
\begin{equation}
  \xymatrix@R-2.3em{
&&\mbox{\scriptsize$0$}&&\mbox{\scriptsize$1$}\\
&&\bullet\ar@/_3ex/[llddddddd]_{j}\ar@/^/[rr]^{a_{1}}\ar@/^4ex/[rr]^{a_{2}}&&\bullet\ar@/^/[ll]^{c_{1}}
\ar@/^4ex/[ll]^{c_2}\ar@/^7ex/[ll]^{\ell_1}\ar@{..}@/^10ex/[ll]\ar@{..}@/^11ex/[ll]
\ar@/^12ex/[ll]^{\ell_{n-2}}& \\ \\ \\ \\ \\&&&&&\mbox{\framebox[1cm]{\begin{minipage}{1cm}\centering $n\geq3$\end{minipage}}} \\ \\
\bullet\ar@/_/[rruuuuuuu]^{i_1}\ar@/_3ex/[rruuuuuuu]_{i_2}\ar@{..}@/_6ex/[rruuuuuuu]\ar@{..}@/_8ex/[rruuuuuuu]\ar@/_10ex/[rruuuuuuu]_{i_{n-1}}&&&&&\\
\mbox{\scriptsize\phantom{'}$0'$}&&&&
}
\end{equation}
For precise definitions and statements we refer to \cite[\S 4]{BBLR} (see in particular Theorem 4.4, {\em loc.~cit.}).

Except for the case $n=2$, the Hilbert scheme $\operatorname{Hilb}^c(\operatorname{Tot}\mathcal{O}_{\mathbb{P}^1}(-n))$ is not  a holomorphic symplectic variety. However, it carries a natural Poisson structure induced by the Poisson bivector defined on the Hirzebruch surface $\Sigma_n$, as follows from Bottacin's general results  \cite{Bot}.
It seems to be an interesting and challenging problem to characterise this Poisson structure in purely quiver-theoretic terms, possibly by resorting
to the noncommutative notion of ``double Poisson bracket'' on the path algebra of a quiver (see \cite{VDB, CrBo, BIE}).
\end{rem}

\bigskip
\section{The minimal case}\label{sectionminimalcase}

In this section we give a complete description of the moduli spaces  $\M^{n}(r,a,c)$ when the minimality condition stated in Theorem
\ref{mainthm}, namely
\begin{equation}\label{eq:min}
 c= C_{\text{m}}(n,a) =\frac{1}{2}na(1-a)\,,
\end{equation}
is satisfied. In the following we will simply write  $C_{\text{m}}$ instead of $C_{\text{m}}(n,a)$, since no ambiguity is likely to arise.

\begin{rem}
If $\E$ is a torsion-free sheaf on $\Sigma_{n}$ which is trivial at infinity and satisfies the ``minimality'' condition \eqref{eq:min}, then $\E$ is locally free. This can be deduced from the canonical injection $\E\rightarrowtail\E^{**}$ (but it is also a consequence of the monadic description \eqref{eqMinimalMonad} below).
\end{rem}

\begin{thm}\label{thm:minimal}
 There are isomorphisms
\begin{equation}
 \mathcal{M}^n\left(r,a,C_{\mathrm{m}} \right)\simeq
  \begin{cases}
   \operatorname{Gr}(a,r)&\text{if $n=1$;}\\
   T^\vee\operatorname{Gr}(a,r)^{\oplus n-1}&\text{if $n\geq2$,}
  \end{cases}
\label{eqIsoM}
\end{equation}
where $\operatorname{Gr}(a,r)$ is the Grassmannian of $a$-planes in $\Com^r$.
\end{thm}
The next three Subsections are devoted to prove Theorem \ref{thm:minimal}.
\subsection{The monad in the minimal case.}
Condition \eqref{eq:min} implies $\Uk=0$, so that the monad \eqref{fundamentalmonad} reduces to the complex
\begin{equation}
\xymatrix{
 0\ar[r]&\Ol_{\Sigma_n}(1,-1)^{\oplus na}\oplus\Ol_{\Sigma_n}^{\oplus r-a}\ar[r]^-{\beta}&\Ol_{\Sigma_n}(1,0)^{\oplus (n-1)a}\ar[r]&0\,,
}
\label{eqMinimalMonad}
\end{equation}
whose cohomology sheaf is just the kernel of $\beta$.  To ensure that this sheaf is trivial at infinity, one has also to impose the invertibility of the linear map $\Phi = H^0(\beta|_{\ell_\infty}(-1))$ (see \cite[\S 3]{BBLR}).

We denote by $\GL(a,r)$ the maximal parabolic subgroup of $\GL(r)$ consisting of upper block triangular matrices of the form
 \[
  \begin{pmatrix}
   A&B\\
   0&C
  \end{pmatrix}\,,\qquad\text{where}\quad
  \begin{cases}
   A\in\operatorname{Mat}_\Com(a\times a)&\\
   B\in\operatorname{Mat}_\Com(a\times(r-a))\\
   C\in\operatorname{Mat}_\Com((r-a)\times (r-a))\,.
  \end{cases}
 \]

First of all, notice that, if $n=1$ or $a=0$, then $\Wk=0$, and $\Phi$ is the identity (i.e.~zero) morphism between null vector spaces: this implies that the variety $\Lk$ reduces to a point, and consequently $\Pk=\GL(r)$. To describe the corresponding moduli spaces we have to compute the automorphism group of $\Vk$:
\begin{itemize}
 \item for $n=1$, $\Vk=\Ol_{\Sigma_1}(1,-1)^{\oplus a}\oplus\Ol_{\Sigma_1}^{\oplus r-a}$; therefore,
$$ \Aut(\Vk)\simeq\GL(a,r)\ \ \text{and}\ \  
 \M^1\left(r,a,a(1-a)/2\right)\simeq \GL(r)/\GL(a,r)\simeq\operatorname{Gr}(a,r)\,;
$$
 \item for $a=0$, $\Vk=\Ol_{\Sigma_n}^{\oplus r}$; therefore,
  $$  \Aut(\Vk)\simeq\GL(r)  \ \ \text{and}\ \  
  \M^n\left(r,0,0\right)\simeq\GL(r)/\GL(r)=\{\ast\}=\operatorname{Gr}(0,r)\,.
$$
\end{itemize}

Let us now assume $n\geq2$ and $a\geq 1$. To prove Theorem \ref{thm:minimal} we make use of the well-known isomorphism
\begin{equation}\label{eq:iso}
 T^\vee\operatorname{Gr}(a,r)^{\oplus n-1}\simeq \left (\Hom(\Com^{r-a},\Com^a)^{\oplus n-1}\times\GL(r)\right )/\GL(a,r)\,,
\end{equation}
where the $\GL(a,r)$-action is given by:
\begin{align*}
 \begin{pmatrix}
  A&B\\
  0&C
 \end{pmatrix}\cdot b&=AbC^{-1}&\text{for $b\in\Hom(\Com^{r-a},\Com^a)$,}\\
  \begin{pmatrix}
  A&B\\
  0&C
 \end{pmatrix}\cdot \theta&=\theta\begin{pmatrix}
  A&B\\
  0&C
 \end{pmatrix}^{-1}&\text{for $\theta\in\GL(r)$.}
\end{align*}

\smallskip
A more explicit description of the bundle $\Pk$ can be obtained by the following procedure.\\
{\bf 1)} 
 We represent the morphism $\beta$ by a matrix. To this aim, we have to choose a basis for the vector space
 \[
  \Hom(\Com^{na},\Com^{(n-1)a})\otimes H^0(\Ol_{\Sigma_n}(0,1))\oplus\Hom(\Com^{r-a},\Com^{(n-1)a})\otimes H^0(\Ol_{\Sigma_n}(1,0))\,.
 \]
We take homogeneous coordinates $[y_1:y_2]$ on $\Pu$ and pull them back to $\Sigma_n$ by means of the canonical projection $\pi\colon\Sigma_n\to\Pu$; the 
elements $\{y_2^qy_1^{h-q}\}_{q=0}^h$ form a basis for the vector space $H^0(\Ol_{\Sigma_n}(0,h))$, for all $h\geq1$. We denote by $s_E$ the unique (up to homotheties) global section of $\Ol_{\Sigma_n}(E)$ and by $s_\infty$ the section of $\Ol_{\Sigma_n}(1,0)$ whose vanishing locus is $\ell_\infty$. The multiplication by $s_E$ induces an immersion
\[
\xymatrix{
  \Ol_{\Sigma_n}(0,n)\ar@{>}[r]&\Ol_{\Sigma_n}(1,0)\,,
 }
\]
so that the set $\{y_2^qy_1^{n-q}s_E\}_{q=0}^{n}\cup\{s_\infty\}$ is a basis for the space $H^0(\Ol_{\Sigma_n}(1,0))$. According to these choices,  the morphism 
$\beta$ is represented by a matrix of the form
\begin{equation}
\begin{pmatrix} \beta_{1} & \beta_{2} \end{pmatrix}= \left(\beta_{10}y_1+\beta_{11}y_2\quad\sum_{q=0}^n\beta_{2q}(y_2^qy_1^{n-q}s_E)+\beta_{2,n+1}s_\infty\right)\,,
\label{eqBeta}
\end{equation}
whilst an automorphism $\psi\in\Aut(\Vk)$ is represented by a matrix of the form
\begin{equation}
\psi=
\begin{pmatrix}
\psi_{11} & \psi_{12}\\
0 & \psi_{11}
\end{pmatrix}=
\begin{pmatrix}
\psi_{11} & \sum_{q=0}^{n-1}\psi_{12,q}(y_2^qy_1^{n-1-q}s_E)\\
0 & \psi_{22}
\end{pmatrix}\,.
\end{equation}
{\bf 2)} 
 The morphism $\Phi$  is represented by an $n(n-1)a\times n(n-1)a$ matrix , whose only nonvanishing  terms are $\beta_{10},\beta_{11}\in\operatorname{Mat}_\Com((n-1)a\times na)$:
\begin{equation}
 \Phi=
 \begin{pmatrix}
  \beta_{10}&&\\
  \beta_{11}&\beta_{10}&\\
  &\beta_{11}&\ddots&\\
  &&\ddots&\beta_{10}\\
  &&&\beta_{11}
 \end{pmatrix}\,.
\label{eqPhi}
\end{equation}
{\bf 3)} A framing for the kernel of $\beta$ is provided by the choice of a basis for $H^0(\ker\beta|_{\ell_\infty})=\ker H^0(\beta|_{\ell_\infty})$. In other words, it is given by an injective linear map
\[
 \xi\colon\Com^r\to H^0(\Vk|_{\ell_{\infty}})\quad\text{such that}\quad H^0(\beta|_{\ell_\infty})\circ\xi=0\,.
\]
Summing up, {\em the bundle $\Pk$ can be described as the set of pairs $(\beta,\xi)$ as above such $\det\Phi\neq 0$}. 

The group $\Gk=\Aut\Vk\times\Aut\Wk$ acts on $\Pk$  as follows:
\[
 (\psi,\chi)\cdot(\beta,\xi)=(\chi\circ\beta\circ\psi^{-1},H^0(\psi|_{\ell_\infty})\circ\xi)\,.
\]

\subsection{A technical Lemma}
Let  $X$ be a smooth algebraic variety over $\Com$ and $G$ a complex affine algebraic group acting on $X$; let 
$\gamma\colon X\times G \to  X\times X$ be the induced morphism given by $(x, g) \mapsto (x, g\cdot x)$.
 The set-theoretical quotient $X/G$ has a natural structure of ringed space, whose topology is the quotient topology induced by the canonical projection $q\colon X\lra X/G$
 and whose structure sheaf is the sheaf of $G$-invariant functions.
 If the action is free and $\gamma$ is a closed immersion,  then $X/G$ is a smooth algebraic variety, the pair $(X/G,q)$ is a geometric quotient of $X$ modulo $G$, and $X$ is a (locally isotrivial) principal $G$-bundle over $X/G$. This can be proved by arguing as in the proof of  \cite[Theorem~5.1]{BBR}.

Let $Y$ be a smooth closed subvariety of $X$ and let $H\stackrel{\iota}{\hookrightarrow}G$ be a closed subgroup of $G$. Assume that $H$ acts on $Y$ and that the inclusion $j\colon Y \hookrightarrow X$ is $H$-equivariant.  We denote by  $p\colon Y\lra Y/H$ the canonical projection.
 \begin{lemma}
 \label{lemmaquot}
 If the intersection of $\im j$ with every $G$-orbit in $X$ is nonempty and,  
for any $G$-orbit $O_{G}$ in $X$, one has $\Stab_{G}(O_{G}\cap\im j)=\im\iota$,
then $j$ induces an isomorphism $\bar{\jmath}\colon Y/H \longrightarrow X/G$ of algebraic varieties.
\end{lemma}
\begin{proof}
By \cite[Prop.~0.7]{Mum}   the morphism $q$ is affine. Hence, if  $U\subset X/G$ is an open affine subset,  $V=q^{-1}(U)$ is  affine as well; if we set $V=\Spec A$, then $U= \Spec(A^{G})$, and the restricted morphism $q|_{V}$ is induced by the canonical immersion $q^{\sharp}\colon A ^{G}\hookrightarrow A$. 
Since $j$ is an affine morphism \cite[Prop.~1.6.2.(i)]{EGA2}, the counterimage $W=j^{-1}(V)$ is affine, $W=\Spec B$, and by the equivariance of $j$ it is $H$-invariant. It follows that its image $p(W)=\Spec(B^{H})$ is affine and the restricted morphism $p|_{W}$ is induced by the canonical immersion $p^{\sharp}\colon B^{H}\hookrightarrow B$. Let $j^{\sharp}\colon A \longrightarrow B $ be the ring homomorphism determined by $j$.  It is not difficult to check
that $q^{\sharp} \circ j^{\sharp}$ is an isomorphism onto $p^{\sharp}(B^H) \subset B$. So $j^{\sharp}$ induces a ring isomorphism 
$\bar{\jmath}\,^{\sharp} \colon A^G \longrightarrow B^H$. This shows that $\bar{\jmath}\colon Y/H \longrightarrow X/G$ is an isomorphism.
\end{proof}

\begin{cor}
Under the previous hypotheses, the principal $H$-bundle $p\colon Y\longrightarrow Y/H$ is a reduction of the principal $G$-bundle $q\colon X\lra X/G$.
In particular, if $X \longrightarrow X/G$  is locally trivial, then $Y\longrightarrow Y/H$ is locally trivial as well.
\end{cor}

\subsection{Reduction of the structure group.}\label{reductionstructuregroup}
The proof of Theorem \ref{thm:minimal} is based on the reduction of the structure group $\Gk$ of the principal bundle $\Pk \lra \mathcal{M}^n(r,a,C_{\mathrm{m}})$ to the group $\GL(a,r)$. To this aim let us introduce 
the closed immersion
\begin{equation}
 j\colon \Hom(\Com^{r-a},\Com^a)^{\oplus n-1}\times\GL(r)\hookrightarrow\Pk
\label{eqj}
\end{equation}
given by $j(b_1,\dots.b_{n-1},\theta)= (\beta,\xi)$, where the morphism $\beta$ is defined by the equations
\begin{equation}
\begin{cases}
  \beta_{10}=
 \begin{pmatrix}
  \mathbf{1}_{(n-1)a}&0
 \end{pmatrix}\,;\\
 \beta_{11}=
 \begin{pmatrix}
  0&\mathbf{1}_{(n-1)a}
 \end{pmatrix}\,;\\
 \beta_{2q}=0&\text{for $q=0,\dots,n$;}\\
 \beta_{2,n+1}=
 \begin{pmatrix}
  b_1\\\vdots\\b_{n-1}
 \end{pmatrix}\,,
\end{cases}
\label{eq-j-Beta}
\end{equation}
\smallskip
and the linear map $\xi$ is represented by the $(n^2a+(r-a))\times r$ matrix
\begin{equation}
 \begin{pmatrix}
  \xi_1\\
  \vdots\\
  (-1)^{i-1}\xi_i\\
  \vdots\\
  (-1)^{n-1}\xi_n\\
  \theta_{r-a}
 \end{pmatrix}\,,
 \quad \text{where} \ \ \begin{cases} \text{$\theta_{r-a}$ is the matrix consisting of the last $r-a$ rows of $\theta$,}\\
  \text{$\xi_i$ is an $na\times r$  block of the form}\quad
 \begin{pmatrix}
  0_{(i-1)a\times r}\\
  \theta^a\\
  0_{(n-i)a\times r}
 \end{pmatrix}\,,  \\  \text{$\theta^a$ is the matrix consisting of the first $a$ rows of $\theta$.}
 \end{cases}
 \label{eqxi}
\end{equation}
A direct computation shows that the pair $(\beta,\xi)$ belongs to $\Pk$.

\smallskip
We claim that \emph{the immersion $j$ determines a reduction of the principal $\Gk$-bundle $\Pk \lra \mathcal{M}^n(r,a,C_{\mathrm{m}})$ to the
principal $\GL(a,r)$-bundle  $\Hom(\Com^{r-a},\Com^a)^{\oplus n-1}\times\GL(r) \lra \mathcal{M}^n(r,a,C_{\mathrm{m}})$}. Indeed we can define an immersion
 $\imath\colon \GL(a,r) {\hookrightarrow} \Gk=\Aut(\Vk)\times\Aut(\Wk)$ by setting
 \begin{equation}
\imath \begin{pmatrix}
  A&B\\
  0&C
 \end{pmatrix} = (\psi,\chi)\,,
\label{eq-imath}
\end{equation}
where
\[
 \psi=
 \begin{pmatrix}
  A\otimes\mathbf{1}_n&B\otimes \mathbf{v}\\
  0&C
 \end{pmatrix}\quad\text{and}\quad\chi=A\otimes\mathbf{1}_{n-1}\,,
\]
with  $\mathbf{v}=(v_0,\dots,v_{n-1})^T$ and $v_i=(-1)^{n-1-i}y_2^iy_1^{n-1-i}s_E$.

It is easy to see that the immersion $j$ is $\GL(a,r)$-equivariant.
\begin{lemma}
The image of $j$ has nonempty intersection with every $\Gk$-orbit, and for any $\Gk$-orbit $O$, the stabilizer in $\Gk$ of $O\cap\im j$ is precisely the image of $\imath$.
\label{eqj-red}
\end{lemma}
\begin{proof}
We consider the following block decomposition of the matrix $\Phi^{-1}$:
\begin{equation}
\begin{aligned}
\Phi^{-1}&=
\begin{pmatrix}
\varphi_{11} & \cdots & \varphi_{1n}\\
\vdots & & \vdots\\
\varphi_{n-1,1} & \cdots & \varphi_{n-1,n}\\
\end{pmatrix}
\begin{array}{l}
\updownarrow na \phantom{\varphi_{11}}\\
\vdots \phantom{\varphi_{11}}\\
\updownarrow na \phantom{\varphi_{11}}
\end{array}\\
&\phantom{=}\mspace{20mu}
\begin{matrix}
\overleftrightarrow{\mbox{\tiny$(n-1)a$}} & \cdots & \overleftrightarrow{\mbox{\tiny$(n-1)a$}} 
\end{matrix}
\end{aligned}
\label{eqPhi-1}
\end{equation}
Because of eq.~\eqref{eqPhi}, the conditions $\Phi\Phi^{-1}=\Phi^{-1}\Phi=\bm{1}_{n(n-1)a}$ imply that
\begin{equation}
\begin{aligned}
\beta_{10}\varphi_{1q}&=\delta_{1,q}\bm{1}_{(n-1)a} & \text{for}\quad q&=1,\dots,n\\
\varphi_{1q}\beta_{10}+\varphi_{1,q+1}\beta_{11}&=\delta_{1,q}\bm{1}_{na} & \text{for}\quad q&=1,\dots,n-1
\end{aligned}
\label{eqbeta-phi}
\end{equation}
Let $(\beta,\xi)\in\Pk$, and let $O$ be its $\Gk$-orbit. Eqs. \eqref{eqbeta-phi} imply that, by acting with $\Gk$, one can find a point inside $O$ such that
\begin{equation}
\beta_{10}=
\begin{pmatrix}
\bm{1}_{(n-1)a} & 0
\end{pmatrix}\,.
\label{eqb10}
\end{equation}
We call $O_{0}$ the subvariety of $O$ cut by eq. \eqref{eqb10}. It is easy to see that the stabilizer $G_{0}$ of $O_{0}$ in $\Gk$ is the subgroup characterized by the condition $\psi_{11}=
\left(\begin{smallmatrix}
\chi & 0\\
0 & A
\end{smallmatrix}\right)$, for some $A\in \GL(a)$.

Let $O_{q}$ be the subvariety of $O_{0}$ cut by the equation
\begin{equation}
\beta_{11}=
\begin{pmatrix}
* & 0\\
0 & \bm{1}_{qa}
\end{pmatrix}
\end{equation}
and let $G_{q}$ be the closed subgroup of $G_{0}$ characterized by the condition $\chi=
\left(\begin{smallmatrix}
* & 0\\
0 & A\otimes\bm{1}_{q}
\end{smallmatrix}\right)$, for $q=1,\dots,n-1$.
By using eqs.~\eqref{eqbeta-phi} and reasoning by induction, one can show that, for any $q=0,\dots,n-2$ and for any point $(\beta,\xi) \in O_q$, there exists an element $g \in G_q$ such that $g\cdot (\beta,\xi)$ lies inside $O_{q+1}$. Moreover, it is easy to check that the stabilizer of $O_{q+1}$ in $G_q$ is $G_{q+1}$.

Summing up, if $(\beta,\xi)\in\Pk$ is any point and $O$ is its $\Gk$-orbit, by acting with $\Gk$ on $(\beta,\xi)$ one can find a point in $O_{n-1}$, which is the subvariety of $O$ cut by eq. \eqref{eqb10} and by the equation
\begin{equation}
\beta_{11}=
\begin{pmatrix}
0 & \bm{1}_{(n-1)a}
\end{pmatrix}\,.
\label{eqb11}
\end{equation}
The stabilizer $G_{n-1}$ of $O_{n-1}$ in $\Gk$ is the subgroup characterized by the conditions $\psi_{11}=A\otimes\bm{1}_{n}$ and $\chi =A\otimes\bm{1}_{n-1}$, for some $A\in \GL(a)$.

Given any point $(\beta,\xi)\in O_{n-1}$, by acting with $G_{n-1}$ one can find a point such that
\begin{equation}
\beta_{2q}=0\qquad\text{for}\quad q=0,\dots,n\,.
\label{eqb2q}
\end{equation}
Let $O_{n}$ be the subvariety cut in $O_{n-1}$ by eq.~\eqref{eqb2q}. It is easy to see that the stabilizer $G_{n}$ of $O_{n}$ in $G_{n-1}$ is the closed subgroup determined by the condition 
$\psi_{12}=B\otimes \mathbf{v}$, where $\mathbf{v}=(v_0,\dots,v_{n-1})^T$ and $v_i=y_2^i(-y_1)^{n-1-i}s_E$, for some $B\in\operatorname{Mat}_\Com(a\times(r-a))$.

It is easy to see that $G_{n}=\im\imath$. Let $Z$ be the subvariety cut in $\Pk$ by eqs.~\eqref{eqb10}, \eqref{eqb11} and \eqref{eqb2q}: we claim that $Z=\im j$. Indeed the condition $H^0(\beta|_{\ell_\infty})\circ\xi$ implies that, for all points $(\beta,\xi)\in Z$, the matrix $\xi$ has the form described in eq.~\eqref{eqxi} and, since $\xi$ has maximal rank, it follows that the $r\times r$ matrix $\left(\begin{smallmatrix} \theta^{a} \\ \theta_{r-a}\end{smallmatrix}\right)$ is invertible.
\end{proof}
Lemmas \ref{lemmaquot} and \ref{eqj-red} imply that the immersion $j$ determines a reduction of the structure group of the  principal $\Gk$-bundle $\Pk$ to a principal $\GL(a,r)$-bundle, as we claimed. In particular, one has the isomorphisms
\begin{equation}
 (\Hom(\Com^{r-a},\Com^a)^{\oplus n-1}\times\GL(r))/\GL(a,r)\simeq\Pk/\Gk\, \simeq \mathcal{M}^n(r,a,C_{\mathrm{m}});
 \label{eqIsoRed}
\end{equation}
 in view of \eqref{eq:iso}, this concludes the proof of Theorem \ref{thm:minimal}.

 We notice that the varieties $\Hom(\Com^{r-a},\Com^a)^{\oplus n-1}\times\GL(r)$ in the case $n>1$ and $\GL(r)$ in the case $n=1$ can be thought of as ADHM
 data spaces for $\mathcal{M}^n(r,a,C_{\mathrm{m}})$. In Section \ref{Section-Nakajima-Flag} the schemes $\mathcal{M}^n(r,a,C_{\mathrm{m}})$ will be constructed as quiver varieties, and  so a different ADHM description will be provided.

\begin{rem}  Theorem \ref{thm:minimal} is consistent with instanton counting, which shows that the moduli spaces $\mathcal{M}^n\left(r,a,C_{\text{m}} \right)$ have the same Betti numbers as $\operatorname{Gr}(a,r)$ \cite{BPT}.
\end{rem}

\subsection{Some geometric remarks.} \label{geometricremarks}
In this subsection, we give a more intrinsic interpretation to the isomorphism \eqref{eqIsoM}.
\begin{prop}
\label{lmEExt}
Let $\E$ be a sheaf on $\Sigma_{n}$.
\begin{enumerate}
\item[(i)]
The sheaf $\E$ is locally free, trivial at infinity, and satisfies condition \eqref{eq:min}
if and only if it fits into an extension of the form
\begin{equation}
\xymatrix{
0 \ar[r] & \On(E)^{\oplus a} \ar[r]^-{i} & \E \ar[r]^-{p} & \On^{\oplus r-a} \ar[r] & 0
}
\label{eqEExt}
\end{equation}
for some integers $r>0$ and $0\leq a<r$.
\item[(ii)]
Two vector bundles $\E$ and $\E'$ which are trivial at infinity and satisfy condition \eqref{eq:min}
are isomorphic if and only if they fit into isomorphic extensions of the form \eqref{eqEExt}.
\item[(iii)] The set of isomorphism classes of vector bundles $\E$ on $\Sigma_{n}$ which are trivial at infinity, satisfy condition \eqref{eq:min} and such that $\rk(\E)=r$, $c_{1}(\E)=aE$ can be identified with the vector space $\Ext^{1}_{\On}\left(\On^{\oplus r-a},\On(E)^{\oplus a}\right)$.
\end{enumerate}
\end{prop}
\begin{proof}
(i)
Let $\E$ be a vector bundle on $\Sigma_{n}$ which is trivial at infinity and satisfies the minimality condition \eqref{eq:min}. Then $c_{1}(\E)=aE$ for some integer $a$ such that $0\leq a<r=\rk(\E)$, while the second Chern class $c_{2}(\E)$ is fixed by eq.~\eqref{eq:min}. Theorem \ref{thmEMon} and eq.~\eqref{eqMinimalMonad} imply that $\E\simeq\ker\beta$.

For $n=1$, one has $\ker\beta= \Ol_{\Sigma_{1}}(E)^{\oplus a}\oplus \Ol_{\Sigma_{1}}^{\oplus r-a}$, so that $\E$ fits into a split extension of the form \eqref{eqEExt}: this proves the first statement in this case.

Let us assume $n\geq2$. We claim that 
\begin{equation}
\ker\beta_{1}\simeq\On(E)^{\oplus a}
\label{eqKerBeta1}
\end{equation}
where 
$$\beta_{1}\colon\On(1,-1)^{\oplus na}\lra\On(1,0)^{\oplus(n-1)a}$$
 is the first block of the matrix $\beta$ (see eq.~\eqref{eqBeta}). Indeed $\On(1,-1)^{\oplus na}\simeq\Com^{a}\otimes\On(1,-1)^{\oplus n}$, $\On(1,0)^{\oplus(n-1)a}\simeq\Com^{a}\otimes \On(1,0)^{\oplus n-1}$, and up to the action of $\Gk$ one has $\beta_{1}=\bm{1}_{a}\otimes f$, where
\begin{equation}
f=\begin{pmatrix} \bm{1}_{n-1}y_{1} & 0 \end{pmatrix} + \begin{pmatrix} 0 & \bm{1}_{n-1}y_{2} \end{pmatrix}\,.
\label{eqFinBeta}
\end{equation}
It follows that $\ker\beta_{1}\simeq(\ker f)^{\oplus a}$. Eq.~\eqref{eqFinBeta} implies that $f$ is surjective, so that $\ker f$ is locally free, of rank $1$ and $c_{1}(\ker f)=H-nF=E$. This proves the claim.

By eq.~\eqref{eqMinimalMonad} $\E$ fits into the following commutative diagram
\begin{equation}
\xymatrix{
 & 0 \ar[d] & 0 \ar[d] &  \\
0 \ar[r] & \On(E)^{\oplus a} \ar[r]^-{i} \ar[d]^-{\kappa} & \E \ar[r]^-{p} \ar[d]^-{h} & \On^{\oplus r-a} \ar[r] \ar@{=}[d] & 0\\
0 \ar[r] & \On(1,-1)^{\oplus na} \ar[r]^-{j} \ar[d]^-{\beta_{1}} & \Vk \ar[r]^-{\pi} \ar[d]^{\beta} & \On^{\oplus r-a} \ar[r] & 0\\
 &\Wk \ar@{=}[r] \ar[d] & \Wk \ar[d] \\
& 0 & 0
}
\label{eqDiaExt}
\end{equation}
where $j=\left(\begin{smallmatrix} \bm{1}_{na} \\ 0  \end{smallmatrix}\right)$, $\pi=\left(\begin{smallmatrix} 0 & \bm{1}_{r-a} \end{smallmatrix}\right)$ and $p=\pi\circ h$. In the left column we have used eq.~\eqref{eqKerBeta1} and the surjectivity of $f$. The morphism $i$ is induced by the other morphisms, and the injectivity of $i$ is a consequence of the injectivity of $j\circ \kappa= h\circ i$. The surjectivity of $p$ is a consequence of the Snake Lemma. This proves the first statement in one direction.

Conversely, let $\E$ be a  sheaf on $\Sigma_{n}$ that fits into eq.~\eqref{eqEExt}. The Chern character of $\E$ is easily computed, and in particular it turns out that $\E$ satisfies condition \eqref{eq:min}. Since $\E$ is an extension of locally free sheaves, it is locally free, and its restriction $\E|_{\li}$ is locally free of the same rank. By restricting \eqref{eqEExt} to $\ell_{\infty}$, by twisting the result by $\Ol_{\li}(-1)$ and by taking cohomology one gets $H^{i}(\E|_{\li}(-1))=0$, $i=0,1$, which implies that $\E|_{\li}$ is trivial.

(ii)
In one direction, the second statement is straightforward. In the other direction, we have to distinguish between the case $n=1$ and the case $n\geq 2$.
For $n=1$, \cite[Lemma~3.1]{BBR} implies that
$$\Ext^{1}_{\Ol_{\Sigma_{1}}}\left(\Ol_{\Sigma_{1}}^{\oplus r-a},\Ol_{\Sigma_{1}}(E)^{\oplus a}\right)=0\,.$$
It follows that all extensions of the form \eqref{eqEExt} split, and this proves the second statement in this case.
For $n\geq2$, if $\E$ and $\E'$ are two isomorphic vector bundles  which are trivial at infinity and satisfy condition \eqref{eq:min}, then the first part of the proof entails  
that $\E$ and $\E'$ fit into extensions of the form \eqref{eqEExt}. By \cite[Lemma~4.7]{BBR} any isomorphism $\Lambda\colon\E\lra\E'$ lifts uniquely to a monads isomorphism $(\psi,\chi)$. The second statement is easily checked by means of the diagram \eqref{eqDiaExt}.

(iii) The statement follows from (i) and (ii).
\end{proof}
%
%

Let $X_{n}=\Sigma_{n}\setminus\li$. This open subset can be naturally regarded as the total space of the line bundle $\Ol_{\Pu}(-n)$.
\begin{lemma}
\label{lmE-iso-away-infty}
Two vector bundles $\E$ and $\E'$ of same Chern character, which are trivial at infinity and satisfy condition \eqref{eq:min}, 
 are isomorphic if and only if their restrictions $\E|_{X_{n}}$ and $\E'|_{X_{n}}$ are isomorphic as $\Ol_{X_{n}}$-modules.
\end{lemma}
\begin{proof}
Let $\rk(\E)=\rk(\E')=r$, $c_{1}(\E)=c_{1}(\E')=aE$.
By Proposition \ref{lmEExt}, it is enough to prove that the natural morphism
\begin{equation}
\Ext^{1}_{\On}\left(\On^{\oplus r-a},\On(E)^{\oplus a}\right)\lra\Ext^{1}_{\Ol_{X_{n}}}\left(\Ol_{X_{n}}^{\oplus r-a},\left(\On(E)|_{X_{n}}\right)^{\oplus a}\right)
\end{equation}
is an isomorphism. This is equivalent to prove that the natural morphism
\begin{equation}
H^{1}(\Sigma_{n},\On(E))\lra H^{1}\left(X_{n},\On(E)|_{X_{n}}\right)
\label{eqRestrE}
\end{equation}
is an isomorphism. 

For any $\On$-module $\F$, we denote by $H^{i}_{\li}(\Sigma_{n},\F)$ its $i$-th cohomology group with supports in $\li$. 
By \cite[Exercise~III.2.3.(e)]{Har} there is an exact sequence
\begin{multline}
\xymatrix{
 H^{0}(X_{n},\Ol_{X_{n}}) \ar[r]^-{\partial_{\Ol}} & H^{1}_{\li}(\Sigma_{n},\On) \ar[r] & H^{1}(\Sigma_{n},\On) \ar[r] &
}\\
\xymatrix{
  \ar[r] & H^{1}(X_{n},\Ol_{X_{n}}) \ar[r] & H^{2}_{\li}(\Sigma_{n},\On) \ar[r] & H^{2}(\Sigma_{n},\On)
}\,.
\end{multline}
Our first claim is that
\begin{equation}
\begin{aligned} &\text{the connecting morphism} \quad 
H^{0}(X_{n},\Ol_{X_{n}})\stackrel{\partial_{\Ol}}{\lra} H^{1}_{\li}(\Sigma_{n},\On) \quad
\text{is surjective} \\
&\text{and}\quad H^{2}_{\li}(\Sigma_{n},\On) =0\,.
\end{aligned}
\label{eqCsuppO}
\end{equation}
We first notice that, by  \cite[Lemma~3.1]{BBR}, $H^{1}(\Sigma_{n},\On)=0$, so that $\partial_{\Ol}$ is surjective. Next, by using \cite[Exercise~III.4.1]{Har}, we have the isomorphism
\begin{equation}
H^{1}(X_{n},\Ol_{X_{n}}) \simeq H^{1}(\Pu,\pi_{*}\Ol_{X_{n}})\,,
\label{eq-van1}
\end{equation}
where $\pi\colon X_{n}\lra\Pu$ is the natural projection. It is not difficult to show by direct computation that $\pi_{*}\Ol_{X_{n}}\simeq\Ol_{\Pu}^{\oplus\mathbb{N}}$, so that
\begin{equation}
H^{1}(\Pu,\pi_{*}\Ol_{X_{n}})\simeq H^{1}(\Pu,\Ol_{\Pu})^{\oplus\mathbb{N}}=0\,.
\label{eq-van2}
\end{equation}
Eqs. \eqref{eq-van1} and \eqref{eq-van2} imply that $H^{1}(X_{n},\Ol_{X_{n}})=0$. Finally, again by \cite[Lemma~3.1]{BBR}, we have that $H^{2}(\Sigma_{n},\On)=0$. The vanishing of $H^{2}_{\li}(\Sigma_{n},\On)$ follows.

An analogous claim to \eqref{eqCsuppO} holds true for the sheaf $\On(E)$, i.e.,
\begin{equation}
\begin{aligned}
&\text{the  morphism}\quad
H^{0}\left(X_{n},\On(E)|_{X_{n}}\right)\stackrel{\partial_{E}}{\lra} H^{1}_{\li}(\Sigma_{n},\On(E))\quad
\text{is surjective}\\
&\text{and}\quad H^{2}_{\li}(\Sigma_{n},\On(E))=0\,.
\end{aligned}
\label{eqCsuppE}
\end{equation}
Indeed, the short exact sequence
$$\xymatrix{ 0 \ar[r] & \On \ar[r] & \On(E) \ar[r] & \Ol_{E}(-n) \ar[r] & 0}$$
induces a commutative diagram
\begin{equation}
\xymatrix{
0 \ar[r] & H^{0}(X_{n},\Ol_{X_{n}}) \ar[r]^-{f} \ar[d]^{\partial_{\Ol}}    & H^{0}\left(X_{n},\On(E)|_{X_{n}}\right) \ar[r] \ar[d]^{\partial_{E}}    & H^{0}\left(X_{n},\Ol_{E}(-n)|_{X_{n}}\right) \ar[d]\\
H^{0}_{\li}(\Sigma_{n},\Ol_{E}(-n)) \ar[r] & H^{1}_{\li}(\Sigma_{n},\On) \ar[r]^-{g} & H^{1}_{\li}(\Sigma_{n},\On(E))  \ar[r] & H^{1}_{\li}(\Sigma_{n},\Ol_{E}(-n))
}
\end{equation}
whose rows are exact.
Since $E\cap\li=\varnothing$, there are inclusions
\begin{equation}
E \subset X_{n}\qquad\text{and}\qquad \li\subset \Sigma_{n}\setminus E\,.
\label{eqInclEli}
\end{equation}
From the first inclusion it follows that $H^{0}\left(X_{n},\Ol_{E}(-n)|_{X_{n}}\right) = H^{0}\left(\Pu,\Ol_{\Pu}(-n)\right)=0$, hence $f$ is an isomorphism. The second inclusion, in view of  \cite[Exercise~III.2.3.(f)]{Har}, implies that 
$$H^{i}_{\li}(\Sigma_{n},\Ol_{E}(-n))=0\qquad\text{for all}\quad i\geq0\,,$$
so that $g$ is an isomorphism. Thus,  $\partial_{E}$ is surjective because this is the case for $\partial_{\Ol}$. The  second inclusion of \eqref{eqInclEli} implies also that
$$H^{i}_{\li}(\Sigma_{n},\On(E))\simeq H^{i}_{\li}(\Sigma_{n},\On)\qquad\text{for all}\quad i\geq0\,.$$
By the last statement in \eqref{eqCsuppO} one has $H^{2}_{\li}(\Sigma_{n},\On(E))=0$. This proves the claim \eqref{eqCsuppE}. 

Finally, the fact that the morphism \eqref{eqRestrE}  is an isomorphism follows from the exact sequence
\begin{multline}
\xymatrix{
H^{0}\left(X_{n},\On(E)|_{X_{n}}\right) \ar[r]^-{\partial_{E}} & H^{1}_{\li}(\Sigma_{n},\On(E)) \ar[r] & H^{1}(\Sigma_{n},\On(E)) \ar[r] & 
}\\
\xymatrix{
\ar[r] & H^{1}\left(X_{n},\On(E)|_{X_{n}}\right) \ar[r] & H^{2}_{\li}(\Sigma_{n},\On(E))\,.
}
\label{eqSeqCsuppE}
\end{multline}
\end{proof}

\begin{cor}
\label{proTipFib}
For $n\geq2$, there is an isomorphism between the typical fibre of the vector bundle $T^\vee\operatorname{Gr}(a,r)^{\oplus n-1}$ and the vector space of 
isomorphism classes of  $\Ol_{X_{n}}$-modules $\E|_{X_{n}}$ obtained by restricting vector bundles $\E$ on $\Sigma_{n}$ that are trivial at infinity, satisfy condition \eqref{eq:min}, and such that $\rk(\E)=r$, $c_{1}(\E)=aE$.
\end{cor}

\begin{proof}
In view of Lemma \ref{lmE-iso-away-infty} it is enough to prove that there is an isomorphism between the typical fibre of the vector bundle $T^\vee\operatorname{Gr}(a,r)^{\oplus n-1}$ and the vector space of isomorphism classes of vector bundles $\E$ on $\Sigma_{n}$ which are trivial at infinity, satisfy condition \eqref{eq:min}, and such that $\rk(\E)=r$, $c_{1}(\E)=aE$.
Notice first that there is an isomorphism of complex vector spaces
\begin{equation}
\Ext^{1}_{\On}\left(\On^{\oplus r-a},\On(E)^{\oplus a}\right)\simeq \Hom(\Com^{r-a},\Com^a)\otimes H^{1}(\On(E))\,.
\label{eqTipFib}
\end{equation}
By the Riemann-Roch formula, since $h^{0}(\On(E))=1$ and $h^{2}(\On(E))=0$ (see \cite[Lemma~3.1]{BBR}), one has $h^{1}(\On(E))=n-1$. The right-hand side of \eqref{eqTipFib} can therefore be regarded as the typical fibre of the vector bundle $T^\vee\operatorname{Gr}(a,r)^{\oplus n-1}$ in view of eq. \eqref{eqIsoRed}. The thesis is then a consequence of Proposition \ref{lmEExt}(iii). 
\end{proof}

As for the Grassmannian $\Gr(a,r)$, this parameterizes equivalence classes of framings. More explicitly, two framings $\theta\colon\E|_{\li}\stackrel{\sim}{\lra}\Ol_{\li}^{\oplus r}$ and $\theta'\colon\E'|_{\li}\stackrel{\sim}{\lra}\Ol_{\li}^{\oplus r}$ are equivalent if and only if there is an isomorphism $\Lambda\colon\E\lra\E'$ such that 
\begin{equation}
\theta'\circ\Lambda\vert_{\ell_{\infty}} = \theta
\label{eqEquivFr}
\end{equation}
(cf. Definition \ref{definframedsheaf}).  Since any framing $\theta$ is an isomorphism between two trivial sheaves on $\li\simeq\Pu$, there is a  1-1 correspondence between the set of framings  and the set of the induced isomorphisms $H^{0}(\theta)\colon H^{0}(\E|_{\li})\lra\Com^{r}$ between global sections. For this reason, eq.~\eqref{eqEquivFr} is equivalent to
\begin{equation}
H^{0}(\theta')\circ H^{0}(\Lambda\vert_{\ell_{\infty}}) = H^{0}(\theta)\,.
\end{equation}
In view of the reduction of the structure group of $\Pk$ made in Subsection \ref{reductionstructuregroup}, we see that, for every element $[(\E,\theta)]\in\M^n\left(r,a,C_{\mathrm{m}} \right)$, the space $H^{0}(\E|_{\li})$ can be identified with a fixed $r$-dimensional subspace $V$ of $H^{0}(\Vk|_{\li})$. In particular, if one computes the matrix $H^{0}(\beta|_{\li})$ by means of eqs.~\eqref{eq-j-Beta}, one gets a fixed isomorphism $V\stackrel{\sim}{\lra}\Com^{r}$: in this way, we can regard a framing as a matrix in  $\GL(r)$. By using the immersion \eqref{eq-imath}, one proves easily that two framings $\theta$ and $\theta'$ are equivalent if and only if $H^0(\theta')g=H^0(\theta)$ for some $g\in\GL(a,r)$. With these identifications in mind, it is then straightforward that the Grassmannian $\Gr(a,r)=\GL(r)/\GL(a,r)$ parameterizes the equivalence classes $[\theta]$ of framings.

For $n\geq2$, the previous construction enables one to reinterpret intrinsically the canonical projection $T^\vee\operatorname{Gr}(a,r)^{\oplus n-1}\lra\Gr(a,r)$ as the map $[(\E,\theta)]\lra[\theta]$.

\section{Nakajima's flag varieties and the spaces $\M^n\left(r,a,C_{\mathrm{m}} \right)$}\label{Section-Nakajima-Flag}
In this section we show (Proposition  \ref{corMshMq}) that there is an isomorphism between the moduli space $\M^n(r,a,C_{\mathrm{m}})$ and the moduli space of the representations of a suitable GF quiver with relations. The proof is based on a straightforward generalization of a result due to Nakajima \cite{Na94, Na96}.

For any positive integer $d$, let $\A_{d}$ be the Dynkin quiver
\begin{equation}
\xymatrix@R-2em@C+1ex{
\mbox{\footnotesize{$0$}} & \mbox{\footnotesize{$1$}} & \mbox{\footnotesize{$2$}} &  & \mbox{\footnotesize{$d-2$}} & \mbox{\footnotesize{$d-1$}}\\
\bullet  & \bullet \ar[l]_-{a_{1}} & \bullet \ar[l]_-{a_{2}}  & \cdots \ar[l]  & \bullet \ar[l] & \bullet \ar[l]_-{a_{d-1}}
}
\end{equation}
So, $I=\{0,\dots,d-1\}$ and the dimension vector $\vec{v}$ of a representation $(V,X)$ of  $\A_{d}$ is an element of $\mathbb{N}^{d}$. We want to associate with $\A_{d}$ a sequence of GF quiver $\Q_{d,n}$ for $n\geq 1$. Since we are interested only in $(\vec{v},\vec{w})$-dimensional representations, where $\vec{w}=(u,0,\dots,0)$ for some integer $u >0$, it is enough to add to $I$ only the vertex $0'$. 
We define the quivers $\Q_{d,n}$ as follows
\begin{equation}
\begin{gathered}
\xymatrix@R-2em@C+1ex{
\mbox{\footnotesize{$0'$}}&\mbox{\footnotesize{$0$}} & \mbox{\footnotesize{$1$}} & \mbox{\footnotesize{$2$}} &  & \mbox{\footnotesize{$d-2$}} & \mbox{\footnotesize{$d-1$}}\\
 \bullet &\bullet \ar[l]_-{j}  & \bullet \ar[l]_-{a_{1}} & \bullet \ar[l]_-{a_{2}}  & \cdots \ar[l]  & \bullet \ar[l] & \bullet \ar[l]_-{a_{d-1}}
}\\
\xymatrix@R-2em@C+1.5em{
\mbox{\footnotesize{$0'$}}&\mbox{\footnotesize{$0$}} & \mbox{\footnotesize{$1$}} & \mbox{\footnotesize{$2$}} &  & \mbox{\footnotesize{$d-2$}} & \mbox{\footnotesize{$d-1$}}\\
 \bullet \ar@/_15pt/[r]^-{i_{1}} \ar@/_30pt/[r]^-{i_{2}} \ar@/_35pt/@{.>}[r] \ar@/_40pt/@{.>}[r]  \ar@/_50pt/[r]_-{i_{n-1}} 
 &\bullet  \ar@/_/[l]_-{j}   \ar@/_15pt/[r]^-{b_{11}} \ar@/_30pt/[r]^-{b_{12}} \ar@/_35pt/@{.>}[r] \ar@/_40pt/@{.>}[r]  \ar@/_50pt/[r]_-{b_{1,n-1}}
 & \bullet \ar@/_/[l]_-{a_{1}}  
  \ar@/_15pt/[r]^-{b_{21}} \ar@/_30pt/[r]^-{b_{22}} \ar@/_35pt/@{.>}[r] \ar@/_40pt/@{.>}[r]  \ar@/_50pt/[r]_-{b_{2,n-1}}  & \bullet \ar@/_/[l]_-{a_{2}}  \ar@/_15pt/[r] \ar@/_30pt/[r] \ar@/_35pt/@{.>}[r] \ar@/_40pt/@{.>}[r]  \ar@/_50pt/[r] &\mspace{10mu} \cdots \mspace{10mu} \ar@/_/[l]  \ar@/_15pt/[r] \ar@/_30pt/[r] \ar@/_35pt/@{.>}[r] \ar@/_40pt/@{.>}[r]  \ar@/_50pt/[r] & \bullet \ar@/_/[l]   \ar@/_15pt/[r]^-{b_{d-1,1}} \ar@/_30pt/[r]^-{b_{d-1,2}} \ar@/_35pt/@{.>}[r] \ar@/_40pt/@{.>}[r]  \ar@/_50pt/[r]_-{b_{d-1,n-1}} & \bullet \ar@/_/[l]_-{a_{d-1}} \\
}
\end{gathered}
\begin{gathered}
\mbox{\framebox[1cm]{\begin{minipage}{1cm}\centering $n=1$\end{minipage}}}
\vspace{3cm}
\\
\mbox{\framebox[1cm]{\begin{minipage}{1cm}\centering $n\geq2$\end{minipage}}}
\end{gathered}
\end{equation}

A representation of a quiver of this form is supported by a vector space $V\oplus U$, where $U$ is associated with the vertex $0'$; for the sake of brevity, we shall say that such a representation is $(\vec{v},u)$-dimensional.

For $n\geq 2$ let $J_{d,n}$ be the ideal of the algebra $\Com\Q_{d,n}$ generated by the following relations:
\begin{equation}
\begin{gathered}
i_{q}j=0\qquad q=1,\dots,n-1 \qquad \qquad\text{when}\quad d=1\\[5pt]
\left\{
\begin{aligned}
a_{1}b_{1q}+i_{q}j&=0\\
b_{1q}a_{1}&=0
\end{aligned}
\right.\qquad q=1,\dots,n-1\qquad\qquad\text{when}\quad d=2\\[5pt]
\left\{
\begin{aligned}
a_{1}b_{1q}+i_{q}j&=0\\
a_{p+1}b_{p+1,q}-b_{pq}a_{p}&=0\\
b_{d-1,q}a_{d-1}&=0
\end{aligned}
\right.\qquad 
\begin{aligned}
p&=1,\dots,d-2\\
q&=1,\dots,n-1
\end{aligned}
\qquad \qquad\text{when}\quad d>2\,.
\end{gathered}
\label{eqGenJ}
\end{equation}
We define the algebra $\F_{d,n}$ as follows
\begin{equation}\label{algebraFDN}
\F_{d,n}=
\begin{cases}
\Com \Q_{d,1} &\qquad\text{when}\quad n=1\\
\Com\Q_{d,n}/J_{d,n}  &\qquad\text{when}\quad n>1\,.\\
\end{cases}
\end{equation}
We choose an integer $u>d$ and a dimension vector $\vec{v}=(v_{0},\dots,v_{d-1})$ such that $u>v_{0}>v_{1}>\dots >v_{d-1}>0$. Let $U=\Com^{u}$ and 
$V= \bigoplus_{i=0}^{d-1} V_i =\bigoplus_{i=0}^{d-1}\Com^{v_{i}}$. We fix also the stability parameter $\vartheta^+=(1,1,\dots,1)\in\mathbb{R}^{d}$.

 We recall that the partial flag variety $\Fl(\vec{v},u)$ is the smooth projective variety whose points can be identified with filtrations $\Com^{u}\supset E_{0} \supset E_{1} \supset \dots \supset E_{d-1}$ of complex vector spaces such that $\dim E_{i}=v_{i}$, $i=0,\dots,d-1$.  The following result was proved by Nakajima for $n\leq 2$ \cite{Na94, Na96}.
\begin{prop}
\label{propCotFlag}
Let $\M(\F_{d,n},\vec{v},u)_{\vartheta^+}$ be the moduli space of $\vartheta^+$\!-semistable $(\vec{v},u)$-dimensional representations of $\F_{d,n}$.
There is an isomorphism
\begin{equation}
\M(\F_{d,n},\vec{v},u)_{\vartheta^+}\simeq
\begin{cases}
\Fl(\vec{v},u) &\qquad\text{if}\quad n=1;\\
T^\vee\Fl(\vec{v},u)^{\oplus n-1}&\qquad\text{if}\quad n\geq2.
\end{cases}
\label{eqIsoFlag}
\end{equation}
\end{prop}
To prove Proposition \ref{propCotFlag} we simplify the notation we introduced in Section \ref{quiversection}. For a representation of $\F_{d,n}$ supported by $\mathbb{V}:=\left(\bigoplus_{i=0}^{d-1}\Com^{v_{i}}\right)\oplus\Com^{u}$ we set
\begin{align}
\label{eqABfe}
e&=X_{j} & f_{q}&=X_{i_{q}} & A_{p}&=X_{a_{p}} & B_{pq}&=X_{b_{pq}}
\end{align}
with $p=1,\dots,d-1$ and $q=1,\dots,n-1$ (in the case $d=1$ there are no morphisms $A_{p}$ and $B_{pq}$, whilst in the case $n=1$ there are no morphisms $f_{q}$ and $B_{pq}$).

According to Definition \ref{defFrStab}, a $\vartheta^+$\!-semistable $(\vec{v},u)$-dimensional representation of $\F_{d,n}$ is defined in terms of an auxiliary quiver $\Q_{d,n}^{u}$ and of an ideal $J^{u}_{d,n}\subset\Com\Q_{d,n}^{u}$. The quiver $\Q_{d,n}^{u}$ is defined by renaming the vertex $0'$ as $\infty$, by replacing the arrow $j$ with $u$ arrows $\tilde{\jmath}_{1},\dots,\tilde{\jmath}_{u}$ and, if $n>1$, by replacing the arrow $i_{q}$ with $u$ arrows $\tilde{\imath}_{q1},\dots,\tilde{\imath}_{qu}$, for all $q=1,\dots,n-1$. For $n>1$ the ideal $J^{u}_{d,n}$ is generated by the relations obtained by replacing the product $i_{q}j$ with the sum of products $\sum_{l=1}^{u}\tilde{\imath}_{ql}\tilde{\jmath}_{l}$ in eqs. \eqref{eqGenJ}.
The definition of the algebra $\F^{u}_{d,n}$  is given analogously to eq.~\eqref{algebraFDN}.

For a representation of $\F^{u}_{d,n}$ supported by $\mathbb{U}:=\left(\bigoplus_{i=0}^{d-1}\Com^{v_{i}}\right)\oplus\Com$ we set
\begin{align}
\tilde{e}_{l}&=X_{\tilde{\jmath}_{l}} & \tilde{f}_{ql}&=X_{\tilde{\imath}_{ql}} & A_{p}&=X_{a_{p}} & B_{pq}&=X_{b_{pq}}
\end{align}
with $l=1,\dots,u$, $p=1,\dots,d-1$, and $q=1,\dots,n-1$  (in the case $d=1$ there are no morphisms $A_{p}$ and $B_{pq}$, whilst in the case $n=1$ there are no morphisms $\tilde{f}_{ql}$ and $B_{pq}$). We now write down the isomorphism \eqref{eqIsoCB0} for the particular case in question. 
Once fixed a basis $\{\varepsilon_{1},\dots,\varepsilon_{u}\}$ for $\Com^{u}$, we define the linear morphisms
\begin{equation}
\begin{array}{rccl}
\varphi_{l}\colon&\Com^{u}&\lra&\Com\\
& z = \sum z_k \varepsilon_k &\lmt & z_{l}
\end{array}\qquad,\qquad
\begin{array}{rccl}
\psi_{l}\colon&\Com&\lra&\Com^{u}\\
& \nu &\lmt & \nu \varepsilon_{l}
\end{array}
\end{equation}
for $l=1,\dots, u$; of course, one has $\id_{\Com^{u}}=\sum_{l=1}^{u}\psi_{l}\circ\varphi_{l}$. The isomorphism \eqref{eqIsoCB0} is given by the map
\begin{equation}
\label{eqIsoCB'}
(e,f_{p},A_{p},B_{pq})_{
\begin{subarray}{l}
p=1,\dots,d-1  \\
q=1,\dots,n-1
\end{subarray}
}
\lmt (\tilde{e}_{l},\tilde{f}_{pl},A_{p},B_{pq})_{
\begin{subarray}{l}
p=1,\dots,d-1  \\
q=1,\dots,n-1\\
l=1,\dots,u
\end{subarray}
}
\end{equation}
where $\tilde{e}_{l}=\varphi_{l}\circ e$ and $\tilde{f}_{pl}=f_{p}\circ\psi_{l}$.

\begin{lemma}
\label{lmStabQdn}
 A representation  $\left(\mathbb{V},X\right) \in \operatorname{Rep}(\F_{d,n};\vec{v},u)$ is 
$\vartheta^+$\!-semistable if and only if it is $\vartheta^+$\!-stable. It is $\vartheta^+$\!-semistable if and only if the morphisms $e$, $A_{0}$,\dots, $A_{d-1}$ are injective.
\end{lemma}
\begin{proof}
Given a subrepresentation $(S\oplus S_\infty,Y)$ of a representation $\left(\mathbb{U},X\right)$ of $\F^{u}_{d,n}$, we set $s_{i}=\dim S_{i}$ for $i=\infty,0,\dots,d-1$ and $\vec{s}=(s_{0},\dots,s_{d-1})$.
According to Definition \ref{defFrStab}, a representation $\left(\mathbb{U},X\right)$ of $\F^{u}_{d,n}$ is $\widehat{\vartheta}^+$\!-semistable if, for all proper, nontrivial subrepresentations $(S\oplus S_\infty,Y)$, one has
\begin{equation}
\vartheta^+\cdot\vec{s}\leq(\vartheta^+\cdot\vec{v})s_{\infty}\,,
\label{eqSemiStab1'}
\end{equation}
it is  $\widehat{\vartheta}^+$\!-stable if strict inequality holds. It is easy to see that a representation $\left(\mathbb{U},X\right)$ is  $\widehat{\vartheta}^+$\!-semistable if and only if for all proper, nontrivial subrepresentations one has $s_{\infty}=1$, and that all $\widehat{\vartheta}^+$\!-semistable representations are stable. This establishes the first statement of the Lemma.

As for the second statement,  we start by showing that, given a representation $\left(\mathbb{U},X\right)$, the condition $s_{\infty}=1$ holds true for all proper, non trivial subrepresentations if and only if
\begin{equation}
\bigcap_{l=1}^{u}\ker \tilde{e}_{l}=\{0\}\qquad\text{and}\qquad\ker A_{p}=\{0\}\qquad\text{for}\quad p=0,\dots,d-1\,.
\label{eqTeA'}
\end{equation}
To prove this claim in one direction we argue by contradiction. If we assume that
eq.~\eqref{eqTeA'} is false, then:\\
$\bullet$ in the case $n=1$, one can find a subrepresentation supported by $S=(0,\dots,0,S_{m},0,\dots,0)$, where 
\begin{equation}\label{eqsubrep}
S_{m} = \begin{cases}
&\bigcap_{l=1}^{u}\ker \tilde{e}_{l} \quad \text{for}\  m=0 \\
&\ker A_{m} \quad \text{for}\  m\in\{1,\dots,d-1\}, \text{with} \ d>1 \,;
\end{cases}\end{equation}
$\bullet$ in the case $n>1$, one can find a subrepresentation supported by $S=(0,\dots,0,S_{m},S_{m+1},\dots,S_{d-1})$, where $S_{m}$ is defined as in eq.~\eqref{eqsubrep}, while $S_{m+1},\dots,S_{d-1}$ are defined inductively by the formula
\begin{equation}
S_{p}=\sum_{q=1}^{n-1}B_{pq}(S_{p-1})
\label{eqSp'}
\end{equation}
where $p=m+1,\dots,d-1$.\\ In both cases, one has $s_\infty = 0$, so that the given subrepresentations are destabilizing.

In the other direction, we assume that  eq.~\eqref{eqTeA'} holds true. Let $(S\oplus S_\infty,Y)$ be a proper, nontrivial subrepresentation. If $s_{\infty}=0$, then $s_{m}>0$ for some $m\in\{0,\dots,d-1\}$, so that  one has the immersion
\begin{equation}
\sum_{l=1}^{u}(\tilde{e}_{l}\circ A_{0}\circ A_{1} \circ\dots\circ A_{m})(S_{m}) \subseteq S_{\infty}
\end{equation}
Eq. \eqref{eqTeA'} implies that $s_{\infty}>0$, and so we have a contradiction. Hence, $s_{\infty} = 1$.

The second statement is then equivalent to eq.~\eqref{eqTeA'}, because one has $e(v)=\sum_{l=1}^{u}\tilde{e}_{l}(v)\varepsilon_{l}$ for all $v\in\Com^{v_{0}}$.
\end{proof}
\begin{proof}[Proof of Proposition \ref{propCotFlag}.]
The cases $n=1, 2$, are, respectively,  Example 3.7 in \cite{Na96} and Theorem 7.3 in \cite{Na94} (see also \cite[Example~3.2.7]{Gin}).

Note that there is a $G_{\vec{v}}$-equivariant morphism $\tilde{q}\colon\operatorname{Rep}(\F_{d,2},\vec{v},u)^{ss}_{\vartheta^+}\lra\operatorname{Rep}(\F_{d,1},\vec{v},u)^{ss}_{\vartheta^+}$ given by the maps
\begin{equation}
\left\{
\begin{array}{rcll}
(e,f) & \lmt & e
&\qquad\text{if $d=1$}\\[7pt]
(e,f,A_{p},B_{p})_{p=1}^{d-1} & \lmt & (e,A_{p})_{p=1}^{d-1}
&\qquad\text{if $d>1$}
\end{array}
\right.
\end{equation}
where, for simplicity, we have put $f=f_{1}$ and $B_{p}=B_{p1}$. The morphism $\tilde{q}$ descends to a morphism $q\colon \M(\F_{d,2},\vec{v},u)_{\vartheta^+}\lra \M(\F_{d,1},\vec{v},u)_{\vartheta^+}$. By composing $q$ with the isomorphisms \eqref{eqIsoFlag} for $n=1,2$, one gets the canonical projection $T^\vee\Fl(\vec{v},u)\lra\Fl(\vec{v},u)$.

For $n\geq 3$, it is easy to prove that there is a $G_{\vec{v}}$-equivariant isomorphism
\begin{equation}
\operatorname{Rep}(\F_{d,n},\vec{v},w)^{ss}_{\vartheta^+}\simeq
\underbrace{\R_{2}\times_{\R_{1}}\R_{2}\times_{\R_{1}}\cdots\times_{\R_{1}}\R_{2}}_{(n-1)-\text{times}}
\end{equation}
where $\R_{2}=\operatorname{Rep}(\F_{d,2},\vec{v},u)^{ss}_{\vartheta^+}$ and $\R_{1}=\operatorname{Rep}(\F_{d,1},\vec{v},u)^{ss}_{\vartheta^+}$. This isomorphism descends to the quotient:
\begin{equation}
\begin{aligned}
\M(\F_{d,n},\vec{v},u)_{\vartheta^+} &\simeq
\underbrace{\M_{2}\times_{\M_{1}}\M_{2}\times_{\M_{1}}\cdots\times_{\M_{1}}\M_{2}}_{(n-1)-\text{times}}\simeq\\
&\simeq\underbrace{T^\vee\Fl(\vec{v},u)\times_{\Fl(\vec{v},u)}T^\vee\Fl(\vec{v},u)\times_{\Fl(\vec{v},u)}\cdots\times_{\Fl(\vec{v},u)}T^\vee\Fl(\vec{v},u)}_{(n-1)-\text{times}}\simeq\\
&\simeq T^\vee\Fl(\vec{v},u)^{\oplus n-1}
\end{aligned}
\end{equation}
where $\M_{2}=\M(\F_{d,2},\vec{v},u)_{\vartheta^+}$ and $\M_{1}=\M(\F_{d,1},\vec{v},u)_{\vartheta^+}$.
\end{proof}
As a direct consequence of Proposition \ref{propCotFlag}, we get a quiver description for the moduli space $\M^n(r,a,C_{\mathrm{m}})$.
\begin{prop}
\label{corMshMq}
There is an isomorphism of complex algebraic varieties
\begin{equation}
\M^n(r,a,C_{\mathrm{m}})\simeq\M(\F_{1,n}, a,r)_{\vartheta^+}
\end{equation}
for any $n\geq1$.
\end{prop}
\begin{proof}
It suffices to compose the isomorphisms \eqref{eqIsoM} and \eqref{eqIsoFlag}.
\end{proof}
In particular, the previous result allows one to regard the quasi-affine variety $\operatorname{Rep}(\F_{1,n}, a,r)^{ss}_{\vartheta^+}$ as a space of ADHM data for $\M^n(r,a,C_{\mathrm{m}})$.

\begin{rem}
The space $\M^2(r,a,C_{\mathrm{m}})$ is isomorphic to the
Nakajima quiver variety $\N_{0,1}(\A_{1},a,r)$ (indeed,  $\Q_{1,2}=\overline{\A_{1}^{\mathrm{fr}}}$). Since the pair $(0,1)$ is $a$-regular (see Definition \ref{defin-v-reg}), $\M^2(r,a,C_{\mathrm{m}})$
carries a symplectic structure which is induced by that one in eq.~\eqref{eq-tildeomega}. With the notation introduced in eq.~\eqref{eqABfe}, this symplectic form can be written as
\begin{equation}
\omega=\tr(\de f_1\wedge\de e) \,.
\label{eq-omega}
\end{equation}
It is easy to see that $\omega$ coincides, up to isomorphism, with the canonical symplectic structure of $T^\vee\Gr(a,r)$. 
\end{rem}

\begin{rem}
As proved by Sala \cite{SAL},  $\M^2(r,a,C_{\mathrm{m}})$ --- like all the spaces $\mathcal{M}^2(r,a,c)$ --- carries also a holomorphic symplectic structure defined in sheaf-theoretic terms. The question that naturally arises is the following: does this symplectic structure coincide with that defined  by eq.~\eqref{eq-omega}?

For $n\neq 2$, the spaces $\M^n(r,a,C_{\mathrm{m}})$ --- and, more generally, all the spaces $\mathcal{M}^n(r,a,c)$ --- are expected to carry a natural Poisson structure. This is suggested by the results proved by Bottacin \cite{Bot95, Bot, Bot00} and by the case $r=1$, but it is unclear how this structure could be constructed.

Work is in progress to address these issues.
\end{rem}

\frenchspacing\bigskip

 \def\cprime{$'$} \def\cprime{$'$} \def\cprime{$'$} \def\cprime{$'$}

\end{document}